\documentclass[english,12pt,reqno]{siamart190516}

\title{Analysis of injection operators in multigrid solvers for hybridized discontinuous Galerkin methods%
\thanks{
This work is supported by the Deutsche Forschungsgemeinschaft (DFG, German Research Foundation) under Germany's Excellence Strategy EXC 2181/1 - 390900948 (the Heidelberg STRUCTURES Excellence Cluster). P.~Lu has been supported by the Alexander von Humboldt Foundation.}}

\headers{Analysis of multigrid injections for HDG methods}{P. Lu, A. Rupp, G. Kanschat}

\author{
Peipei Lu\thanks{Department of Mathematics Sciences, Soochow University, Suzhou, 215006, China,
\email{pplu@suda.edu.cn}}
\and
Andreas Rupp
\thanks{Interdisciplinary Center for Scientific Computing (IWR), Heidelberg University, Mathematikon, Im Neuenheimer Feld 205, 69120 Heidelberg, Germany, \email{andreas.rupp@fau.de},
\email{kanschat@uni-heidelberg.de}}
\and
Guido Kanschat\footnotemark[3]}

\usepackage{amssymb}

\usepackage[english]{babel}
\usepackage[utf8]{inputenc}
\usepackage{tikz,tikzscale}
\usetikzlibrary{calc}
\tikzset{>=latex}
\usepackage{xcolor}
\usepackage{enumerate}
\usepackage{booktabs,multirow}

\newsiamremark{remark}{Remark}

\newcommand{\elem}{\ensuremath{T}}
\newcommand{\mesh}{\ensuremath{\mathcal T}}

\newcommand{\faceSet}{\ensuremath{\mathcal F}}
\newcommand{\faceSetDir}{\ensuremath{\mathcal F^\textup D}}
\newcommand{\face}{\ensuremath{F}}
\newcommand{\skeletal}{\ensuremath{\Sigma}}
\newcommand{\skeletalSpace}{\ensuremath{M}}

\newcommand{\contThreeElementSpace}[1]{\ensuremath{V^\textup c_{#1,p+3}}}
\newcommand{\contFourElementSpace}[1]{\ensuremath{V^\textup c_{#1,p+4}}}
\newcommand{\linElementSpace}{\ensuremath{\overline V^\textup c}}
\newcommand{\discElementSpace}{\ensuremath{V}}
\newcommand{\polynomials}{\ensuremath{\mathcal P}}
\newcommand{\level}{\ensuremath{\ell}}
\newcommand{\iterMgOuter}{i}
\newcommand{\iterMgInner}{m}

\newcommand{\Div}{\nabla\!\cdot\!}

\newcommand{\extensionOp}{\ensuremath{\mathcal U^\textup c}}
\newcommand{\averagingOp}{\ensuremath{I^\textup{avg}}}
\newcommand{\linearInterpolation}{\overline{I}}

\newcommand{\injectionOp}{\ensuremath{I}}
\newcommand{\projectionOp}{\ensuremath{P}}
\newcommand{\projectionOrthogonalOP}{\ensuremath{\Pi}}

\newcommand{\skeletalProj}{\ensuremath{\Pi^\partial}}
\newcommand{\contLinProj}{\ensuremath{\overline \Pi^\textup c}}
\newcommand{\discProj}{\ensuremath{\Pi^\textup d}}
\newcommand{\liftingOp}{\ensuremath{S}}

\renewcommand{\vec}[1]{\ensuremath{\boldsymbol{#1}}}
\newcommand{\Nu}{\ensuremath{\vec \nu}}
\newcommand{\dx}{\ensuremath{\, \textup d x}}
\newcommand{\ds}{\ensuremath{\, \textup d \sigma}}
\newcommand{\localU}{\ensuremath{\mathcal U}}
\newcommand{\localQ}{\ensuremath{\vec{\mathcal Q}}}

\newcommand{\avg}[1]{\{\!\{ #1 \}\!\}}
\newcommand{\jump}[1]{{[\![ #1 ]\!]}}

\newcommand{\ureconstructed}{\overline{u}}

\begin{document}

\maketitle

\begin{abstract}
 Uniform convergence of the geometric multigrid V-cycle is proven for HDG methods with a new set of assumptions on the injection operators from coarser to finer meshes. The scheme involves standard smoothers and local solvers which are bounded, convergent, and consistent. Elliptic regularity is used in the proofs. The new assumptions admit injection operators local to a single coarse grid cell. Examples for admissible injection operators are given. The analysis applies to the hybridized local discontinuous Galerkin method, hybridized Raviart-Thomas, and hybridized Brezzi-Douglas-Marini mixed element methods. Numerical experiments are provided to confirm the theoretical results.
\end{abstract}



\section{Introduction}
%
While hybridizable discontinuous Galerkin (HDG) methods have become an active research area in the last decades, there are few results concerning the efficient solution of the resulting algebraic systems even for second order elliptic problems. Multigrid methods for HDG so far have relied either on continuous coarse spaces or injection operators with unnecessarily wide stencils. In this paper, we analyze fundamental assumptions on injection operators for the convergence of the V-cycle algorithm and we propose several such operators which are strictly local to a single coarse grid cell.


With multigrid methods for HDG, the difficulty of devising an ``injection operator'' originates from the fact that the finer mesh has edges which are not refinements of the edges of the coarse mesh. In \cite{TanPhD}, several possible injection operators were discussed, but most of them turned out to be unstable. Following these results, a heterogeneous multigrid method with continuous coarse space was developed in~\cite{CockburnDGT2013}. Alternatively, \cite{FabienKMR19} applies a $p$-multigrid method to obtain a system which is equivalent to a face-centered finite volume discretization, and uses an $h$-multigrid method, afterwards. In \cite{LuRK2020}, we constructed a  multigrid method which is homogeneous in the sense that it employs the same HDG method on all levels.
We proved uniform convergence of the method for stabilization parameters $\tau_\level \sim{h_\level^{-1}}$ under the assumption of elliptic regularity. We used the injection operator from~\cite{ChenLX2014}. It is higher order accurate depending on the polynomial degrees, but involves wide stencils resulting from interpolation into continuous subspaces.

In this paper, we analyze the V-cycle multigrid method under abstract assumptions on the injection operator. These assumptions are met by the previously used continuous injection, but allow for a much wider class. In particular, injection operators must be conforming with continuous subspaces, but are not required to interpolate into these.
Additionally, our new analysis only assumes the more general condition  $\tau_\level h_\level \lesssim 1$ on the stabilization parameter. In particular, the convergence analysis of multigrid methods for the hybridized Raviart-Thomas (RT-H) method and the hybridized Brezzi-Douglas-Marini (BDM-H) method is covered.

The arguments in this article are centered around three sets of assumptions. There are assumptions on the local solvers of the HDG methods, labelled (LS1)--(LS6) in subsection~\ref{sec:ass-local}. These concern stability, consistency, and convergence of the method in several norms, and all of them have been proven in previous publications. Second are our new assumptions on the injection operators, namely~\eqref{EQ:IA1} and~\eqref{EQ:IA2} in subsection~\ref{sec:ass-injection}. These two sets enable us to prove assumptions~\eqref{EQ:precond1} to~\eqref{EQ:precond3} in subsection~\ref{SEC:main_convergence_result}. These are from the article~\cite{DuanGTZ2007} and they are sufficient for uniform multigrid convergence.

The remainder of this paper is structured as follows:
First, we introduce the considered model and some general assumptions. Afterwards, Section \ref{SEC:injection} discusses possible injection operator, and Section \ref{SEC:multigrid} describes the multigrid method and its main convergence result. The preliminaries of this main result are shown to hold true (under our general assumptions) in Section \ref{SEC:analysis}. A list of HDG methods which are covered by our analysis is provided in Section \ref{SEC:possible_methods} together with sources of the proofs of their properties. The remainder of the paper consists of numerical experiments in Section \ref{SEC:numerics} and a discussion of the achieved results of the publication.
%
\section{Model equation and discretization}
%
We consider the standard diffusion equation in mixed form defined on a polygonally bounded Lipschitz domain $\Omega \subset \mathbb R^d$ with boundary $\partial \Omega$. We assume homogeneous Dirichlet boundary conditions on $\partial\Omega$. Thus, we approximate solutions $(u, \vec q)$ of
\begin{subequations}\label{EQ:diffusion_mixed}
\begin{align}
 \Div \vec q & = f && \text{ in } \Omega,\\
 \vec q + \nabla u & = 0 && \text{ in } \Omega,\\
 u & = 0 && \text{ on } \partial \Omega,
\end{align}
\end{subequations}
for a given function $f$. In the analysis, we will assume elliptic regularity, namely $u \in H^2(\Omega)$ if $f \in L^2(\Omega)$, such there is a constant $c>0$ for which holds
\begin{gather}
  |u|_{H^2(\Omega)}
  \le c \| f \|_{L^2(\Omega)}.
\end{gather}
Here and in the following, $L^2(\Omega)$ denotes the space of square integrable functions on $\Omega$ with inner product and norm
\begin{equation}
 (u,v)_0 := \int_\Omega u v \dx, \qquad \text{and} \qquad \| u \|^2_0 := (u,u)_0.
\end{equation}
The space $H^k(\Omega)$ is the Sobolev space of $k$-times weakly
differentiable functions with derivatives in $L^2(\Omega)$ with seminorm $\left|\cdot\right|_{H^k(\Omega)}$. The broken Sobolev space on the mesh $\mesh$ is $H^k(\mesh)$ with semi-norm $\left|\cdot\right|_{k,\mesh}$.
We note that the assumption of homogeneous boundary data was introduced for simplicity of presentation and can be lifted by standard arguments.
%
\subsection{Spaces for the HDG multigrid method}
%
Starting out from a subdivision $\mesh_0$ of $\Omega$ into simplices,
we construct a hierarchy of meshes $\mesh_\level$ for
$\level = 1,\dots, L$ recursively by refinement, such that each cell
of $\mesh_{\level-1}$ is the union of several cells of mesh
$\mesh_\level$. We assume that the mesh is regular, such that each facet
of a cell is either a facet of another cell or on the
boundary. Furthermore, we assume that the hierarchy is shape regular
and thus the cells are neither anisotropic nor otherwise distorted.
We call $\level$ the level of the quasi-uniform mesh $\mesh_\level$ and denote by
$h_\level$ the characteristic length of its cells. We assume that
refinement from one level to the next is not too fast, such that there
is a constant $c_\text{ref} > 0$ with
\begin{equation}\label{EQ:cref}
  h_\level \ge c_\text{ref} h_{\level-1}.
\end{equation}
This condition holds obviously for bisection as well as regular refinement.

By $\faceSet_\level$ we denote the set of faces of $\mesh_\level$.
The subset of faces on the boundary is
\begin{gather}
  \faceSetDir_\level := \{\face \in \faceSet_\level : \face \subset \partial \Omega \}.
\end{gather}
Moreover, we define
$\faceSet^\elem_\level := \{ \face \in \faceSet_\level : \face \subset
\partial \elem \}$ as the set of faces of a cell $\elem\in\mesh_\level$.
On the set of faces, we define the space $L^2(\faceSet_\ell)$ as the space of
square integrable functions with the inner product
\begin{gather}
  \langle \lambda, \mu \rangle_\level
  := \sum_{\elem \in \mesh_\level} \frac{|\elem|}{|\partial \elem|}
  \int_{\partial \elem} \lambda \mu \ds \cong  h_\level \sum_{\face \in \faceSet_\level} \int_{\face} \lambda \mu \ds.
\end{gather}
Note that interior faces appear twice in this definition such that expressions like $\langle u, \mu \rangle_\level$ with possibly discontinuous $u \in H^1(\mesh_\level)$ and $\mu \in L^2(\faceSet)$ are defined without further ado. Additionally, this inner product scales with $h_\level$ like the $L^2$-inner product in the bulk domain. Its induced norm is defined by $ \| \mu \|^2_\level = \langle \mu, \mu \rangle_\level$.

Let $p\ge 1$ and $\polynomials_p$ be the space of (multivariate)
polynomials of degree up to $p$. Then, we define the space of piecewise polynomial functions on the skeleton by
\begin{gather}
  \skeletalSpace_\level := \left\{ \lambda \in L^2 (\faceSet_\level) \;\middle|\;
    \begin{array}{r@{\,}c@{\,}ll}
  \lambda_{|\face} &\in& \polynomials_p & \forall \face \in \faceSet_\level\\
  \lambda_{|\face} &=& 0 & \forall \face \in \faceSetDir_\level    
    \end{array}
  \right\}.
\end{gather}

The HDG method involves local spaces $V_\elem$ and $\vec W_\elem$ and a local solver on each mesh cell
$\elem \in \mesh_\level$, producing cellwise approximations $u_\elem \in V_\elem$
and and $\vec q_\elem\in \vec W_\elem$ of the functions $u$ and $\vec q$ in
equation~\eqref{EQ:diffusion_mixed}, respectively. We will also use the concatenations of the spaces $V_\elem$
and $\vec W_\elem$, respectively, as a function space on $\Omega$, namely
\begin{gather}
  \label{EQ:dg_spaces}
  \begin{aligned}
    \discElementSpace_\level
    &:=\bigl\{ v \in L^2(\Omega)
    & \big|\;v_{|\elem} &\in V_\elem,
    &\forall \elem &\in \mesh_\level \bigr\},\\
    \vec W_\level
    &:=\bigl\{ \vec q \in L^2(\Omega;\mathbb R^d)
    & \big|\;\vec q_{|\elem} &\in \vec W_\elem,
    &\forall \elem &\in \mesh_\level \bigr\}.    
  \end{aligned}
\end{gather}
%
\subsection{Hybrid discontinuous Galerkin method for the diffusion equation}\label{SEC:HDG_definition}
%
The HDG scheme for~\eqref{EQ:diffusion_mixed} on a mesh $\mesh_\level$
consists of a local solver and a global coupling equation. The local
solver is defined cellwise by a weak formulation
of~\eqref{EQ:diffusion_mixed} in the discrete spaces
$V_\elem \times \vec W_\elem$ and defining suitable numerical traces and fluxes. Namely, given
$\lambda \in \skeletalSpace_\level$ find $u_\elem \in V_\elem$ and
$\vec q_\elem \in \vec W_\elem$ , such that
\begin{subequations}\label{EQ:hdg_scheme}
\begin{align}
  \int_\elem \vec q_\elem \cdot \vec p_\elem \dx - \int_\elem u_\elem \Div \vec p_\elem \dx
  & = - \int_{\partial \elem} \lambda \vec p_\elem \cdot \Nu \ds
    \label{EQ:hdg_primary}
  \\
  \int_{\partial \elem} ( \vec q_\elem \cdot \Nu + \tau_\level u_\elem ) v_\elem \ds - \int_\elem \vec q_\elem \cdot \nabla v_\elem \dx
  & = \tau_\level \int_{\partial \elem} \lambda v_\elem \ds \label{EQ:hdg_flux}
\end{align}
\end{subequations}
hold for all $v_\elem \in V_\elem$, and all $\vec p_\elem \in \vec W_\elem$, and for
all $\elem \in \mesh_\level$. Here,
$\Nu$ is the outward unit normal with respect to $\elem$ and $\tau_\level > 0$
is the penalty coefficient. While the local solvers are implemented
cell by cell, it is helpful for the analysis to combine them by
concatenation. Thus, the local solvers define a mapping
\begin{gather}
  \begin{split}
    \skeletalSpace_\level & \to \discElementSpace_\level \times \vec W_\level\\
   \lambda &\mapsto (\localU_\level \lambda, \localQ_\level \lambda),
 \end{split}
\end{gather}
such that on each cell $\elem\in \mesh_\level$ holds
$\localU_\level \lambda = u_\elem$ and
$ \localQ_\level \lambda = \vec q_\elem$. In the same way, we define
operators $\localU_\level f$ and $ \localQ_\level f$ for
$f\in L^2(\Omega)$, where now the local solutions are defined by
the system
\begin{subequations}\label{EQ:hdg_f}
  \begin{align}
    \int_\elem \vec q_\elem \cdot \vec p_\elem \dx - \int_\elem u_\elem \Div \vec p_\elem \dx
    & = 0
      \label{EQ:hdg_f_primary}
    \\
    - \int_\elem \vec q_\elem \cdot \nabla v_\elem \dx  + \int_{\partial \elem} ( \vec q_\elem \cdot \Nu + \tau_\level u_\elem ) v_\elem \ds
    & =  \int_{\elem} f v_\elem \dx.
      \label{EQ:hdg_f_flux}
\end{align}
\end{subequations}

Once $\lambda$ has been computed, the HDG approximation
to~\eqref{EQ:diffusion_mixed} on mesh $\mesh_\level$ will be computed as
\begin{equation}
    u_\level = \localU_\level \lambda + \localU_\level f, \qquad
    \vec q_\level = \localQ_\level \lambda + \localQ_\level f
\end{equation}

The global coupling condition is derived through a discontinuous
Galerkin version of mass balance and reads: Find
$\lambda \in \skeletalSpace_\level$, such that for all
$ \mu \in \skeletalSpace_\level$
\begin{equation}
  \sum_{\elem \in \mesh_\level}
  \sum_{\face \in \faceSet^\elem_\level \setminus \faceSetDir_\level}
   \int_\face \left( \vec q_\level \cdot \Nu
    + \tau_\level (u_\level - \lambda)\right) \mu \ds = 0.\label{EQ:hdg_global}
\end{equation}

In~\cite{CockburnGL2009}, it is shown that $(\lambda, u_\level, \vec q_\level) \in \skeletalSpace_\level \times V_\level \times \vec W_\level$ is the solution of the coupled system~\eqref{EQ:hdg_scheme}---\eqref{EQ:hdg_global} if and only if it is the solution of
\begin{subequations}\label{EQ:hdg_condensed}
\begin{equation}\label{EQ:hdg_condensed_forms}
 a_\level (\lambda, \mu) = b_\level(\mu) \qquad \forall \mu \in \skeletalSpace_\level,
\end{equation}
with
\begin{align}
 a_\level(\lambda, \mu) = & \int_\Omega \localQ_\level \lambda \localQ_\level \mu \dx + \sum_{\elem \in \mesh_\level} \int_{\partial \elem} \tau_\level (\localU_\level \lambda - \lambda) (\localU_\level \mu - \mu) \ds \label{EQ:bilinear_condensed},\\
 b_\level(\mu) = & \int_\Omega \localU_\level \mu f \dx.
\end{align}
\end{subequations}
Furthermore, the bilinear form $a_\level(\lambda, \mu)$ is symmetric and positive definite.
Thus, it induces a norm
\begin{equation}
  \| \mu \|^2_{a_\level} = a_\level(\mu, \mu),
\end{equation}

We close this subsection by associating an operator
$A_\ell\colon \skeletalSpace_\level \to \skeletalSpace_\level$ with
the bilinear form $a_\level(\cdot,\cdot)$ by the relation
\begin{equation}\label{EQ:def_A}
 \langle A_\level \lambda, \mu \rangle_\level = a_\level(\lambda, \mu) \qquad \forall \mu \in \skeletalSpace_\level.
\end{equation}
\begin{remark}
 Setting $\tau_\level \equiv 0$ in the definition of the HDG methods yields hybridized versions of classical mixed methods. Namely, for $\vec W_\elem = [\polynomials_p]^d + \vec x \polynomials_p$ and $\discElementSpace_\elem = \polynomials_p$ we obtain the hybridized Raviart--Thomas (RT-H) method. If $\vec W_\elem = [\polynomials_p]^d$ and $\discElementSpace_\elem = \polynomials_{p-1}$, this defines the hybridized Brezzi--Douglas--Marini (BDM-H) method. In this sense HDG methods are a generalization of these methods and therefore they are covered by our analysis.
\end{remark}
%
\subsection{Operators for the multigrid method and analysis}

After the discrete operator $A_\level$ has been characterized, we introduce the remaining operators here. First, there is an injection operator $\injectionOp_\level \colon \skeletalSpace_{\level - 1} \to \skeletalSpace_\level$. Properties of $\injectionOp_\level$ that ensure our analytical results as well as possible choices of injection operators are presented below in Section~\ref{SEC:injection}.
Next, there are two operators from $\skeletalSpace_{\level}$ to $\skeletalSpace_{\level-1}$, which replace the $L^2$-projection and the Ritz projection of conforming methods, respectively. They are $\projectionOrthogonalOP_{\level-1}$  and $\projectionOp_{\level-1}$ defined by the conditions
\begin{xalignat}3
 \projectionOrthogonalOP_{\level-1}&\colon \skeletalSpace_\level \to \skeletalSpace_{\level-1},
 &\langle \projectionOrthogonalOP_{\level-1} \lambda, \mu \rangle_{\level-1}
 &= \langle \lambda, \injectionOp_\level \mu\rangle_\level
 && \forall \mu \in \skeletalSpace_{\level-1}.
 \label{EQ:L2_projection_definition}
 \\
 \projectionOp_{\level-1}&\colon \skeletalSpace_\level \to \skeletalSpace_{\level-1},
 &a_{\level-1}(\projectionOp_{\level-1} \lambda, \mu)
 &= a_\level(\lambda, \injectionOp_\level \mu)
 && \forall \mu \in \skeletalSpace_{\level-1},
 \label{EQ:projection_definition}
\end{xalignat}
The operator $\projectionOrthogonalOP_{\level-1}$ (or a discrete variation of it) is used in the implementation, while $\projectionOp_{\level-1}$ is key to the analysis.

For the sake of analysis, we also introduce the $L^2$-projections
\begin{align}
 \skeletalProj_\level \colon & H^2(\Omega) \cap H^1_0(\Omega) \to \skeletalSpace_\level, && \langle \skeletalProj_\level u , \mu \rangle_\level = \langle u, \mu \rangle_\level & \forall \mu \in \skeletalSpace_\level,\\
 \discProj_\level \colon & H^1(\Omega) \to \discElementSpace_\level, && (\discProj_\level u, w)_0 = (u,w)_0 & \forall w \in \discElementSpace_\level,
\end{align}
These projections obviously satisfy the standard $H^1$-stability and $L^2$-approximation properties
\begin{xalignat}2
 \| u - \skeletalProj_\level u \|_\level \lesssim & h_\level^2 |u|_2,  && \forall u \in H^2(\Omega),\label{EQ:bdrH2_approx}\\
 \| u - \discProj_\level u \|_0 \lesssim & h_\level |u|_{1},  && \forall u \in H^1(\Omega).\label{EQ:L2H1_approx}
\end{xalignat}
Here and in the following, $\lesssim$ has the meaning of smaller than or equal to up to a constant independent of the mesh size $h_\ell$ or the multigrid level $\ell$. Moreover, we set
\begin{equation}
 \contThreeElementSpace{\level} := \{ u \in H^1_0(\Omega) \colon u|_\elem \in \polynomials_{p+3}(\elem) \; \forall \elem \in \mesh_\level \}.
\end{equation}

The multigrid operator for preconditioning $A_\level$ will be defined in
  Section \ref{SEC:multigrid_algortith}. It will be referred to as
  \begin{gather}
    B_\level\colon \skeletalSpace_\level \to \skeletalSpace_\level.
 \end{gather}
 It relies on a smoother
 \begin{gather}
   R_\level: \skeletalSpace_\level \to \skeletalSpace_\level,
 \end{gather}
  which can be defined in terms of Jacobi or Gauss-Seidel
  iterations, respectively. Denote by $R_\level^\dagger$ the adjoint operator of
  $R_\level$ with respect to
  $\langle \cdot, \cdot \rangle_\level$ and define $R_\level^\iterMgOuter$ by
 \begin{equation}
  R_\level^\iterMgOuter = \begin{cases} R_\level & \text{ if } \iterMgOuter \text{ is odd,} \\ R_\level^\dagger & \text{ if } \iterMgOuter \text{ is even.} \end{cases}
\end{equation}
%
\subsection{Assumptions on local solvers}
\label{sec:ass-local}
%
We assume that the local problem \eqref{EQ:hdg_scheme} satisfies the following conditions for all $\mu \in \skeletalSpace_\level$:
\begin{itemize}
 \item The trace of the local reconstruction $\localU_\level \mu$ approximates the skeletal function $\mu$ itself, namely
 \begin{equation}
  \| \localU_\level \mu - \mu \|_\level \lesssim h_\level \| \localQ_\level \mu \|_0. \tag{LS1}\label{EQ:LS1}
 \end{equation}
 \item Both $\localQ_\level \mu$ and $\localU_\level \mu$ are bounded by the traces:
 \begin{equation}
  \| \localQ_\level \mu \|_0 \lesssim h^{-1}_\level \| \mu \|_\level \quad \text{ and } \quad \| \localU_\level \mu \|_0 \lesssim \| \mu \|_\level. \tag{LS2}\label{EQ:LS2}
 \end{equation}
 \item The reconstruction $\localQ_\level \mu$ approximates the negative gradient of $\localU_\level \mu$. That is,
 \begin{equation}
  \| \localQ_\level \mu + \nabla \localU_\level \mu \|_0 \lesssim h_\level^{-1} \| \localU_\level \mu - \mu \|_\level. \tag{LS3}\label{EQ:LS3}
 \end{equation}
 \item Consistency with the standard linear finite element method in the sense that for  $w \in \linElementSpace_\level$ and $\mu = \gamma_\level w$ there holds
 \begin{equation}
  \localQ_\level \mu = - \nabla w \quad \text{ and } \quad \localU_\level \mu = w. \tag{LS4}\label{EQ:LS4}
 \end{equation}
Here, $\gamma_\level$ is the trace operator mapping sufficiently smooth functions on the domain $\Omega$ to their trace on the skeleton $\skeletal_\level$ and
\begin{equation}
 \linElementSpace_\level := \bigl\{ u \in H^1_0(\Omega) \; \big| \; u_{|\elem} \in \polynomials_1(\elem) \;\; \forall \elem \in \mesh_\level\bigr\}.
\end{equation}
 \item Convergence of the Lagrange multipliers to the projected traces of the analytical solution. That is, if $\lambda$ is the skeletal function of HDG approximation of $u \in H^2(\Omega)$, which itself solves \eqref{EQ:diffusion_mixed}, we have
 \begin{equation}
  \| \skeletalProj_\level u - \lambda \|_\level \lesssim h_\level^2 | u |_2. \tag{LS5}\label{EQ:LS5}
 \end{equation}
 \item The standard spectral properties of the condensed stiffness matrix hold in the sense that
 \begin{equation}
  C_1 \| \mu \|^2_\level \le a_\level(\mu,\mu) \le C_2 h^{-2}_\level \| \mu \|^2_\level. \tag{LS6}\label{EQ:LS6}
 \end{equation}
\end{itemize}

A list of hybrid methods matching these assumptions can be found in Section \ref{SEC:possible_methods}.
%
\section{Injection operators}\label{SEC:injection}
%
We discuss injection operators from two points of view: first, we introduce conditions on such operators which allow us to prove multigrid convergence. As a particular consequence of these conditions, we obtain a quasi-orthogonality condition at the end of the first subsection.
In the second subsection, we present examples for injection operators and prove that the conditions apply.

\subsection{Assumptions on injection operators}
\label{sec:ass-injection}
The purpose of this article is the abstraction from specific injection operators as they have been defined in previous publications. To this end, we will prove convergence of the standard $V$--cycle multigrid method with an injection operator $\injectionOp_\level$ admitting the following assumptions:
\begin{enumerate}
 \item Stability of the injection operator:
 \begin{equation}\tag{IA1}\label{EQ:IA1}
  \| \injectionOp_\level \lambda \|_\level \lesssim \| \lambda \|_{\level - 1} \qquad \forall \lambda \in \skeletalSpace_{\level - 1}.
 \end{equation}
 \item Trace identity for conforming linear finite elements:
 \begin{equation}\tag{IA2}\label{EQ:IA2}
  \injectionOp_\level \gamma_{\level - 1} w = \gamma_\level w \qquad \forall w \in \linElementSpace_{\level - 1}.
 \end{equation}
\end{enumerate}
Trace identity~\eqref{EQ:IA2} for the injection operator together with the consistency~\eqref{EQ:LS4} of the local solver with linear conforming elements yields
\begin{lemma}[Quasi-orthogonality]\label{LEM:quasi_orth}
 Assuming \eqref{EQ:IA2} and \eqref{EQ:LS4}, then for any $\lambda \in \skeletalSpace_{\level}$ there holds:
 \begin{equation}
  (\localQ_\level \lambda - \localQ_{\level - 1} \projectionOp_{\level - 1} \lambda, \nabla w)_0 = 0 \qquad \forall w \in \linElementSpace_{\level - 1}.
 \end{equation}
\end{lemma}
\begin{proof}
 For $w \in \linElementSpace_{\level - 1}$ let $\mu = \gamma_{\level - 1} w$, we have by \eqref{EQ:IA2} and \eqref{EQ:LS4} that
 \begin{equation}
  \injectionOp_\level \mu = \gamma_\level w, \qquad \localQ_{\level - 1} \mu = \localQ_\level \injectionOp_\level \mu = - \nabla w, \qquad \localU_{\level - 1} \mu = \localU_\level \injectionOp_\level \mu = w,
 \end{equation}
 which implies the result.
\end{proof}
\subsection{Possible injection operators}
In this section, we describe several injection operators for which assumptions~\eqref{EQ:IA1} and~\eqref{EQ:IA2} hold. They are labeled $\injectionOp_\level^0$ to $\injectionOp_\level^3$ and their performance in a multigrid method is tested experimentally in Section~\ref{SEC:numerics}.

The assumptions were modeled after the analysis of the ``continuous'' injection operator $\injectionOp^0_\level$, which was introduced in~\cite{ChenLX2014}, where also its stability was proven in the sense of Lemma \ref{TH:conv_result}.
In~\cite[Sect.~2.3]{LuRK2020}, 
convergence of the same multigrid method as in this present article with injection operator $\injectionOp_\level^0$ was proven under the additional assumption  $\tau_\level = \tfrac{c}{h_\level}$. Here, this analysis is generalized to $\tau_\level h_\level \lesssim 1$.

The operator $\injectionOp_\level^0$ is constructed as follows: For $\lambda\in \skeletalSpace_{\level-1}$, apply a \emph{continuous extension operator} defined cellwise: For finite elements defined by Lagrange interpolation (sometimes called nodal elements), we assign the following value at the interpolation points $\vec x$:
\begin{equation}
  [\extensionOp_{\level-1} \lambda](\vec x)
  = \begin{cases}
   \avg{\lambda(\vec x)} & \text{ if $\vec x$ is on the boundary of a face,} \\ \lambda(\vec x) & \text{ if $\vec x$ is in the interior of a face,} \\
   [\localU_{\level-1} \lambda](\vec x) & \text{ if $\vec x$ is in the interior of a cell,} \end{cases}
 \end{equation}
 Here, $\avg{\lambda (\vec x)}$ denotes the arithmetic mean of the values of $\lambda$ from all faces sharing the point $\vec x$.
 Then, $\injectionOp_\level^0\lambda$ is computed by taking the trace on the skeleton of this continuous function, namely
\begin{equation}
  \label{EQ:def-I0}
  \injectionOp^0_\level \lambda = \gamma_\level\extensionOp_{\level-1} \lambda \qquad  \forall \face \in \faceSet_\level.
 \end{equation}

 A simplified version of this operator has been used for the embedded discontinuous Galerkin (EDG) method in \cite{LuRK2021}. The injection operator $\injectionOp_\level^0$ has a wide stencil. In fact, its values on a single fine grid cell involve values from \emph{all} coarse cells sharing a vertex with the parent of the fine cell. This is detrimental for an efficient implementation and, as we see in the numerical experiments, also for convergence of the method. Therefore, we investigate more local injection operators.
 
 We will discuss these operators in the context of regular refinement (also known as red refinement) and of bisection (see Figure~\ref{fig:refinement}).
 \begin{figure}
     \centering
     \includegraphics[width=.8\textwidth]{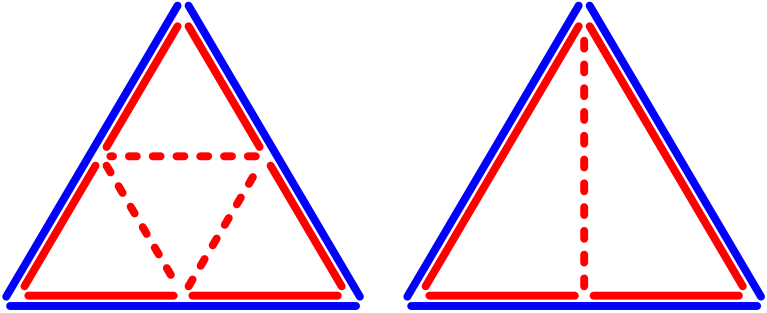}
     \caption{Examples of original faces (blue) faces of refined mesh (red) for regular refinement (left) and bisection (right). Faces which are refinements are solid and newly introduced faces are dashed.}
     \label{fig:refinement}
 \end{figure}
When the mesh $\mesh_{\level-1}$ is refined to obtain $\mesh_\level$, this refinement results in two classes of faces: those which are refinements of the coarser faces and those which are introduced inside the coarser cells. The following three injection operators all act as the embedding operator on the first set. They differ in the assignment of values to the new faces, but they all have in common, that only the values of $\lambda$ on the boundary of the cell enclosing the new face are involved. Thus, they can be deemed \emph{local} injection operators.

The injection operator $\injectionOp_\level^1$ is defined by interpolation only. It is particularly simple for regular refinement in two dimensions, where each new face has its end points on a face of $\mesh_{\ell-1}$. For such a face $\face$ we name these points $\vec a$ and $\vec b$.
With this information, we can define $\injectionOp_\level^1$ on $\face$ by linear interpolation
 \begin{equation}
   \label{EQ:def-I1}
  [\injectionOp_\level^1 \lambda] (\vec x) = \frac{|\vec x - \vec b| \lambda (\vec a) + |\vec x - \vec a| \lambda (\vec b)}{|\vec a - \vec b|},
 \end{equation}
 where $|\cdot|$ denotes the Euclidean norm. For bisection, one point ($\vec a$) is on a coarse face and the other ($\vec b)$ is at a vertex. Thus, the method must be modified such that instead of $\lambda(\vec b)$, we use the average $\avg{\lambda(\vec b)}$.
 
 For higher dimensions, this method can be extended similarly to linear polynomials and higher order methods. For regular refinement, we only need to assign the values to the vertices of new faces, which will be done by averaging.
 
 The injection operator $\injectionOp_\level^2$ uses the local reconstruction of the primary unknown $u$ to assign values to new faces. This results in
 \begin{equation}
 \label{EQ:def-I2}
 \injectionOp_\level^2 \lambda_{|\face} = \localU_{\level - 1} \lambda_{|\face},
 \end{equation}
 for any newly generated face $\face$. This method works the same way in any space dimension and independently of the refinement pattern.
 
 For the third injection operator $\injectionOp_\level^3$, we combine the previous two approaches. If the finite element on the face is defined by Lagrange interpolation, the values of $\injectionOp_\level^3\lambda$ on the boundary of a new face $\face$ are obtained like with $\injectionOp_\level^1$, while they are chosen as for $\injectionOp_\level^2$ in the interior.
 Thus, in two dimensions, for all support points (nodes) $\vec x \in \face$, we set
 \begin{equation}
 \label{EQ:def-I3}
  [\injectionOp_\level^3 \lambda] (\vec x)
  = \begin{cases}
  \lambda(\vec x) & x\in \partial \face,\\
  [\localU_{\level - 1} \lambda] (\vec x) &\text{otherwise}.
  \end{cases}
 \end{equation}
 This methods inherits all the complications mentioned in the description of $\injectionOp_\level^1$.
 
\begin{lemma}
  Injection operator $\injectionOp_\level^1$ is bounded in the sense of~\eqref{EQ:IA1}.
  Assuming additionally \eqref{EQ:LS2}, injection operators  $\injectionOp_\level^0$, $\injectionOp_\level^2$, and $\injectionOp_\level^3$ admit assumption~\eqref{EQ:IA1}.
  Assumption~\eqref{EQ:IA2} holds for $\injectionOp_\level^1$. If, additionally, \eqref{EQ:LS4} holds, then also  $\injectionOp_\level^0$, $\injectionOp_\level^2$, and $\injectionOp_\level^3$ admit assumption~\eqref{EQ:IA2}.
 \end{lemma}
 \begin{proof}
    First, we prove assumption~\eqref{EQ:IA1} for injection operators $\injectionOp^*_\level$ where $*=1,2,3$.
    Let $\elem \in \mesh_{\level-1}$ be a mesh cell which is refined into new cells, and let $\skeletal_\level(\elem)$ be the resulting set of faces. Obviously, for $\lambda\in\skeletalSpace_{\level-1}$ there holds
    \begin{gather}
        \left\| \injectionOp^*_\level \lambda \right\|_{L^2(\partial\elem)}
        = \left\| \lambda \right\|_{L^2(\partial\elem)},
    \end{gather}
    for all three injection operators. For interior faces $\face$, there holds
    \begin{gather}
        \left\| \injectionOp^*_\level \lambda \right\|_{L^\infty(\face)}
        \lesssim \left\| \lambda \right\|_{L^\infty(\partial\elem)},
    \end{gather}
    either by the interpolation property or by the boundedness of $\localU_{\level-1}$. On $\partial\elem$, we can use an inverse estimate, such that we obtain the bound in $L^2$,
    \begin{gather}
        \left\| \injectionOp^*_\level \lambda \right\|_{L^2(\face)}
        \lesssim \left\| \lambda \right\|_{L^2(\partial\elem)}.
    \end{gather}
    The number of faces in $\skeletal_\level(\elem)$ is uniformly bounded by virtue of~\eqref{EQ:cref}. Thus, we obtain
    \begin{gather}
    \label{EQ:injection-lemma-1}
        \left\| \injectionOp^*_\level \lambda \right\|_{L^2(\skeletal_\level(\elem))}
        \lesssim \left\| \lambda \right\|_{L^2(\partial\elem)},
    \end{gather}
    which easily transforms to
    \begin{gather}
        \left\| \injectionOp^*_\level \lambda \right\|_\level
        \lesssim \left\| \lambda \right\|_{\level-1}.
    \end{gather}
    Assumption~\eqref{EQ:IA1} holds for  $\injectionOp_\level^0$ with similar arguments. 
    The additional averaging results in replacing $\partial T$ by $(\partial W(T))$ where $W(T)=\{T'| T'\cap T\neq \emptyset\}$
    in~\eqref{EQ:injection-lemma-1}.
    
    For  $\injectionOp^0_\level$, $\injectionOp_\level^2$ and  $\injectionOp_\level^3$ assumption~\eqref{EQ:IA2} follows from the same property of the local solver, namely~\eqref{EQ:LS4}, while for  $\injectionOp_\level^1$, it is obvious.
 \end{proof}
%
\section{Multigrid method and main convergence result}\label{SEC:multigrid}
%
We consider a standard (symmetric) V-cycle multigrid method applied to the system of linear equations arising from \eqref{EQ:hdg_condensed}. We follow the common approach of treating smoothing and coarse grid corrections separately. As our focus is on the injection operator, we almost entirely ignore the question of smoothing and allow any smoother which fits into the framework of~\cite{BrambleP1992}, in particular pointwise Jacobi and Gauss-Seidel methods. The method in its standard form, see~\cite{BramblePX1991} is presented in Section \ref{SEC:multigrid_algortith} and the analysis based on abstract arguments following~\cite{DuanGTZ2007} follows in Section \ref{SEC:main_convergence_result}.
%
\subsection{Multigrid algorithm}\label{SEC:multigrid_algortith}
%
Let $\iterMgInner \in \mathbb N \setminus \{ 0 \}$ be the number of fine-level smoothing steps. We recursively define the multigrid operator of the refinement level $\level$
\begin{equation}
 B_\level \; : \quad \skeletalSpace_\level \to \skeletalSpace_\level,
\end{equation}
by the following steps. Let $B_0 = A^{-1}_0$. For $\level > 0$, let $x^0 = 0 \in \skeletalSpace_\level$. Then for $\mu\in\skeletalSpace_\level$,
\begin{enumerate}
 \item Define $x^\iterMgOuter \in \skeletalSpace_\level$ for $\iterMgOuter = 1, \ldots, \iterMgInner$ by
 \begin{equation}
  x^\iterMgOuter = x^{\iterMgOuter-1} + R_\level^{\iterMgOuter} ( \mu - A_\level x^{\iterMgOuter-1} ).
 \end{equation}
 \item Set $y^0 = x^\iterMgInner + \injectionOp_\level q$, where $q \in \skeletalSpace_{\level-1}$ is defined as
 \begin{equation}
  q = B_{\level-1} \projectionOrthogonalOP_{\level-1} ( \mu - A_\level x^\iterMgInner).
 \end{equation}
 \item Define $y^\iterMgOuter \in \skeletalSpace_\level$ for $\iterMgOuter = 1, \ldots, \iterMgInner$ as
 \begin{equation}
  y^\iterMgOuter = y^{\iterMgOuter - 1} + R^{\iterMgOuter+\iterMgInner}_\level ( \mu - A_\level y^{\iterMgOuter-1} ).
\end{equation}
\item  Let $B_\level \mu = y^{\iterMgInner}$.
\end{enumerate}
%
\subsection{Main convergence result}\label{SEC:main_convergence_result}
%
The analysis of the multigrid method is based on the framework
introduced in~\cite{DuanGTZ2007}. There, convergence is traced back to
three assumptions. Let $\underline \lambda^A_\level$ be the largest eigenvalue of $A_\level$, and
\begin{gather}
 K_\level := \bigl(1 - (1 - R_\level A_\level) (1 - R^\dagger_\level A_\level)\bigr) A^{-1}_\level.
\end{gather}
Then, there exists constants $C_1, C_2, C_3 > 0$ independent of the mesh level $\level$, such that there holds
\begin{itemize}
\item Regularity approximation assumption:
  \begin{equation}\label{EQ:precond1}
    | a_\level(\lambda - \injectionOp_\level \projectionOp_{\level-1} \lambda, \lambda) |
    \le C_1 \frac{\| A_\level \lambda \|^2_\level}{\underline \lambda^A_\level} \qquad \forall \lambda \in \skeletalSpace_\level. \tag{A1}
  \end{equation}
\item Stability of the ``Ritz quasi-projection'' $\projectionOp_{\level-1}$ and injection $\injectionOp_\level:$
 \begin{equation}\label{EQ:precond2}
  \| \lambda - \injectionOp_\level \projectionOp_{\level-1} \lambda\|_{a_\level} \le C_2 \| \lambda \|_{a_\level} \qquad \forall \lambda \in \skeletalSpace_\level. \tag{A2}
\end{equation}
\item Smoothing hypothesis:
 \begin{equation}\label{EQ:precond3}
  \frac{\| \lambda \|^2_\level}{\underline \lambda^A_\level} \le C_3 \langle K_\level \lambda, \lambda \rangle_\level. \tag{A3}
 \end{equation}
\end{itemize}
Theorem~3.1 in~\cite{DuanGTZ2007} reads
\begin{theorem}\label{TH:main_theorem}
 Assume that \eqref{EQ:precond1}, \eqref{EQ:precond2}, and \eqref{EQ:precond3} hold. Then for all $\level \ge 0$,
 \begin{equation}
  | a_\level ( \lambda - B_\level A_\level \lambda, \lambda ) | \le \delta a_\level(\lambda, \lambda),
 \end{equation}
 where
 \begin{equation}
  \delta = \frac{C_1 C_3}{\iterMgInner - C_1 C_3} \qquad \text{with} \qquad \iterMgInner > 2 C_1 C_3.
 \end{equation}
\end{theorem}

Thus, in order to prove uniform convergence of the multigrid method, we will now set out to verify these assumptions.

Our analysis focuses to standard smoothers like Jacobi or Gauss-Seidel methods.
Thus, proof of the smoothing hypothesis \eqref{EQ:precond3} reduces to verify the assumptions of~\cite[Theorems 3.1 and 3.2]{BrambleP1992}, which in turn boils down to checking the limited interaction property. This holds trivially, since the bilinear form $a_\ell(.,.)$ couples to degrees of freedom only if they are either associated to the same face or to another face which shares a common mesh cell. We note that this also extends to block variants of these smoothers grouping degrees of freedom locally by face or by cell.
%
\section{Convergence analysis}\label{SEC:analysis}
%
We are proving the missing two assumptions of the multigrid error analysis. We begin with a subsection which derives a fundamental theorem on the energy stability of the injection  operators based on the abstract assumptions. Assumption~\eqref{EQ:precond2} is an immediate consequence of these bounds. In the second part of this section, we prove assumption~~\eqref{EQ:precond1}.
%
\subsection{Energy stability of the injection and proof of~\eqref{EQ:precond2}}
%
We begin this part by proving the boundedness of $\injectionOp_\level$ with respect to several measures. In particular, we show energy stability of the injection operator and the ``Ritz projection''.
\begin{lemma}\label{TH:conv_result}
 Assuming \eqref{EQ:IA1}, \eqref{EQ:IA2}, \eqref{EQ:LS1}--\eqref{EQ:LS4}, we have for all $\lambda \in \skeletalSpace_{\level - 1}$
 \begin{align}
  \| \localQ_\level \injectionOp_\level \lambda \|_0 ~\lesssim~ & \| \localQ_{\level - 1} \lambda \|_0, \\
  \| \localU_{\level - 1} \lambda - \localU_\level \injectionOp_\level \lambda \|_0 ~\lesssim~ & h_{\level - 1} \| \localQ_{\level - 1} \lambda \|_0 \lesssim h_\level \| \localQ_{\level - 1} \lambda \|_0.
 \end{align}
 If additionally $\tau_\level h_\level \lesssim 1$, we have
 \begin{align}
  a_\level ( \injectionOp_\level \lambda, \injectionOp_\level \lambda ) & \lesssim a_{\level - 1} (\lambda, \lambda) && \forall \lambda \in \skeletalSpace_{\level - 1},\\
  a_{\level - 1} (\projectionOp_{\level - 1} \lambda, \projectionOp_{\level - 1} \lambda) & \lesssim a_{\level} (\lambda, \lambda) && \forall \lambda \in \skeletalSpace_\level.
 \end{align}
\end{lemma}
In order to prove Lemma \ref{TH:conv_result} we need some preliminaries, which are the subjects of the following lemmas. They make use of the averaging linear interpolation
\begin{gather}
  \averagingOp_\level\colon\discElementSpace_\level \to \linElementSpace_\level,
\end{gather}
defined by the canonical interpolation operator $\linearInterpolation$ into linear finite elements and averaging in the vertices $\vec x$ of the cell $\elem_1$, namely
\begin{equation}
  \left[\averagingOp_\level u\right] (\vec x)
  = \frac1{n_{\vec x}}\sum_{i=1}^{n_{\vec x}} u_{|\elem_i}(\vec x).
\end{equation}
Here, $n_{\vec x}$ is the number of elements meeting in vertex $\vec x$ and $u_{|\elem_i}$ is the restriction of a function $u\in \discElementSpace_\level$ to cell $\elem_i$, which is single valued at $\vec x$. For $\vec x\in \partial\Omega$, we let $\left[\averagingOp_\level u\right] (\vec x) = 0$.
\begin{lemma}\label{LEM:avg_bound_by_localQ}
 If \eqref{EQ:LS1} and \eqref{EQ:LS3} hold, then for any $\level$ there holds
 \begin{align}
  | \averagingOp_\level \localU_\level \lambda |_{1} ~\lesssim~ & \| \localQ_\level \lambda \|_0, && \forall \lambda \in \skeletalSpace_\level,\label{EQ:avg_stability} \\
  \| \localU_\level \lambda - \averagingOp_\level \localU_\level \lambda\|_\level ~\lesssim~ & h_\level  \| \localQ_\level \lambda \|_0 && \forall \lambda \in \skeletalSpace_\level.\label{EQ:avg_error}
 \end{align}
\end{lemma}
\begin{proof}
By standard scaling arguments, we have for the canonical interpolation operator $\linearInterpolation$ interpolating a polynomial $u$ on any cell $\elem$
 \begin{gather}
   \label{EQ:average-1}
    \begin{split}
    h_T^{-2}\| \linearInterpolation u\|_{0,\elem}^2  +  \left|\linearInterpolation u\right|_{1,\elem}^2
     &\lesssim h_T^{-2} |\elem| \sum_{\vec x} |u(\vec x)|^2,
     \\
     h_T^{-2}\| u-\linearInterpolation u\|_{0,\elem}^2 + \left|\linearInterpolation u\right|_{1,\elem}^2
     &\lesssim |u|_{1,\elem}^2.
    \end{split}
 \end{gather}
 A simple computation yields
 \begin{gather}
     \label{EQ:average-2}
     \left[\averagingOp_\level u\right]_{|\elem_1} (\vec x)
  = u_{|\elem_1} (\vec x) + \frac1{n_{\vec x}}\sum_{i=2}^{n_{\vec x}}
  \bigl(u_{|\elem_i}(\vec x)-u_{|\elem_1}(\vec x)\bigr).
 \end{gather}
 Using the fact that $u_{|\elem}$ is polynomial for any $\elem$ and a standard scaling argument, we obtain for any two cells $\elem_i$ and $\elem_j$ sharing an edge
 \begin{gather}
   \label{EQ:average-3}
     \bigl(u_{|\elem_i}(\vec x)-u_{|\elem_j}(\vec x)\bigr)^2
     \lesssim \frac1{|\face|} \int_{\face} \bigl(u_{|\elem_i}-u_{|\elem_j}\bigr)^2.
 \end{gather}
 If two cells sharing the vertex $\vec x$ do not share a face, they are connected by a chain of less than $n_{\vec x}$ cells where each is sharing a face with the next.
 
 At this point, we use the assumption that our meshes are quasi-uniform and form a shape-regular family. As a consequence, $n_{\vec x}$ is uniformly bounded and $|\elem|/|\face| \sim h_\level$. We sum up and use estimate \eqref{EQ:average-2} and \eqref{EQ:average-3} to obtain
 \begin{align*}
     \|\averagingOp_\level u-\linearInterpolation u\|_0^2 \lesssim 
     \sum_{\elem \in \mesh_\level}\sum_{\vec x} |\elem| |\averagingOp_\level u(\vec x)-u(\vec x)|^2 \lesssim \| \jump{u} \|_\level^2,
 \end{align*}
which together with triangle inequality and \eqref{EQ:average-1} indicates
 \begin{gather}
     \|\averagingOp_\level u- u\|_0^2 \lesssim h_{\level}^2  \left|u\right|^2_{1,\mesh_\level}+\| \jump{u} \|_\level^2.
 \end{gather}
  Entering $u = \localU_\level \lambda$ and adding $\lambda-\lambda$ into the second term yields
 \begin{align*}
     \|\averagingOp_\level \localU_\level \lambda- \localU_\level \lambda\|_0^2 \lesssim h_{\level}^2  \left|\localU_\level \lambda\right|^2_{1,\mesh_\level}+\left\|\localU_\level\lambda - \lambda\right\|_\level^2 \lesssim  h_\level^2  \| \localQ_\level \lambda \|_0^2,
 \end{align*}
 where we use \eqref{EQ:LS1} and \eqref{EQ:LS3} in the second inequality. Similarly, we can get
 \begin{align*}
   \left|\averagingOp_\level \localU_\level \lambda- \localU_\level \lambda\right|_{1,\mesh_\level} \lesssim   \| \localQ_\level \lambda \|_0.
 \end{align*}
 Using \eqref{EQ:LS1} and \eqref{EQ:LS3} again, \eqref{EQ:avg_stability} is concluded.
\end{proof}
\begin{lemma}\label{LEM:diff_bound_by_q}
 If \eqref{EQ:LS1} and \eqref{EQ:LS3} hold, then
 \begin{equation*}
  \| \lambda - \gamma_\level \averagingOp_\level \localU_\level \lambda \|_\level \lesssim h_\level \| \localQ_\level \lambda \|_0 \qquad \forall \lambda \in \skeletalSpace_\level.
 \end{equation*}
\end{lemma}
\begin{proof}
 First, we use triangle inequality to split up the difference:
 \begin{equation}
  \| \lambda - \gamma_\level \averagingOp_\level \localU_\level \lambda \|_\level ~\le~ \underbrace{ \| \lambda - \localU_\level \lambda \|_\level }_{ =: \Xi_1 } + \underbrace{ \| \localU_\level \lambda - \gamma_\level \averagingOp_\level \localU_\level \lambda \|_\level }_{ =: \Xi_2 }
 \end{equation}
 Now, we have to bound the individual summands:
 \begin{align}
  \Xi_1 ~\lesssim~ & \| \lambda - \localU_\level \lambda \|_\level \overset{\eqref{EQ:LS1}}\lesssim h_{\level} \| \localQ_\level \lambda \|_0,\\
  \Xi_2 ~\lesssim~ & h_{\level} \| \localQ_\level \lambda \|_0,
 \end{align}
 where $\Xi_2$ is estimated by the means of Lemma \ref{LEM:avg_bound_by_localQ}.
\end{proof}
Using these results, we can prove Lemma \ref{TH:conv_result}.
\begin{proof}[Proof of Lemma \ref{TH:conv_result}]
 We start with proving the first inequality by
 \begin{align*}
  \| \localQ_\level \injectionOp_\level \lambda \|_0 ~\le~ & \| \localQ_\level \injectionOp_\level \lambda + \nabla \averagingOp_{\level - 1} \localU_{\level - 1} \lambda \|_0 + \| \nabla \averagingOp_{\level - 1} \localU_{\level - 1} \lambda \|_0\\
  \lesssim & \| \localQ_\level \injectionOp_\level \lambda - \localQ_\level \gamma_\level \averagingOp_{\level - 1} \localU_{\level - 1} \lambda \|_0 + | \averagingOp_{\level - 1} \localU_{\level - 1} \lambda |_{1},
 \end{align*}
 where the second estimate uses \eqref{EQ:LS4} to replace the gradient by $\localQ_\level \gamma_\level$. Now, by \eqref{EQ:LS2} we can bound the error in $\localQ_\level$ gaining a negative power of $h_\level$, and by \eqref{EQ:IA2} we can insert the injection operator into the subtrahend:
 \begin{align*}
  \| \localQ_\level \injectionOp_\level \lambda \|_0 ~\lesssim~ & h_\level^{-1} \| \injectionOp_\level \lambda - \injectionOp_\level \gamma_{\level - 1} \averagingOp_{\level - 1} \localU_{\level - 1} \lambda \|_\level + | \averagingOp_{\level - 1} \localU_{\level - 1} \lambda |_{1}.
 \end{align*}
  Next, we can get rid of the injection operators due to their $L^2$ stability \eqref{EQ:IA1}. The last step in this chain of inequalities consists in the application of Lemmas \ref{LEM:avg_bound_by_localQ} and \ref{LEM:diff_bound_by_q} to bound the remaining summands:
 \begin{equation}
         \| \localQ_\level \injectionOp_\level \lambda \|_0
         \lesssim h_\level^{-1} \| \lambda - \gamma_{\level - 1} \averagingOp_{\level - 1} \localU_{\level - 1} \lambda \|_{\level-1} + | \averagingOp_{\level - 1} \localU_{\level - 1} \lambda |_{1}\lesssim  \| \localQ_{\level - 1} \lambda \|_0.
 \end{equation}
 
 With the first inequality done, we can also prove the second inequality splitting it up via
 \begin{equation}
  \| \localU_{\level - 1} \lambda - \localU_\level \injectionOp_\level \lambda \|_0
  \le \underbrace{ \| \localU_{\level - 1} \lambda - \averagingOp_{\level - 1} \localU_{\level  - 1} \lambda \|_0 }_{ =: \Xi_1 }
  + \underbrace{ \| \averagingOp_{\level - 1} \localU_{\level  - 1} \lambda - \localU_\level \injectionOp_\level \lambda \|_0 }_{ =: \Xi_2 }
 \end{equation}
 and bounding its individual terms by
 \begin{align*}
  \Xi_1 = & \| \localU_{\level - 1} \lambda - \localU_{\level - 1} \gamma_{\level - 1} \averagingOp_{\level - 1} \localU_{\level  - 1} \lambda \|_0 \\
  \lesssim & \| \lambda - \gamma_{\level - 1} \averagingOp_{\level - 1} \localU_{\level  - 1} \lambda \|_{\level - 1} \lesssim h_{\level} \| \localQ_{\level - 1} \lambda \|_0,
 \end{align*}
  where the identity is assumption \eqref{EQ:LS4}, and the first estimate is the stability of $\localU_{\level-1}$ (i.e. assumption \eqref{EQ:LS2}). The last inequality would be Lemma \ref{LEM:diff_bound_by_q} if we were on level $\level$. Since this is not true, we additionally need to estimate $h_{\level-1}$ by $h_\level$ using \eqref{EQ:cref}. Similar to $\Xi_1$, we use assumption \eqref{EQ:LS4} to rewrite it and \eqref{EQ:LS2} to get rid of $\localU_\level$. Afterwards, however we use the identity property of the injection operator \eqref{EQ:IA2} to insert it into the subtrahend:
 \begin{align*}
  \Xi_2 = & \| \localU_\level \injectionOp_\level \lambda - \localU_\level \gamma_\level \averagingOp_{\level - 1} \localU_{\level  - 1} \lambda \|_0 \\
  \lesssim & \| \injectionOp_\level \lambda - \injectionOp_\level \gamma_{\level - 1} \averagingOp_{\level - 1} \localU_{\level  - 1} \lambda \|_\level\\
  \lesssim & \| \lambda - \gamma_{\level - 1} \averagingOp_{\level - 1} \localU_{\level  - 1} \lambda \|_{\level - 1},
 \end{align*}
 where the stability of the injection operator \eqref{EQ:IA1}, on the other side, allowed us to remove both (the original and the previously inserted) injection operators. As for $\Xi_2$, the remaining steps consist of the application of Lemma \ref{LEM:diff_bound_by_q} and \eqref{EQ:cref}.
 \begin{align*}
  \Xi_2 \lesssim & h_{\level - 1} \| \localQ_{\level - 1} \lambda \|_0 \lesssim h_\level \| \localQ_{\level - 1} \lambda \|_0.
 \end{align*}
 The third inequality is a direct consequence of \eqref{EQ:LS1} (bounding the term $\tfrac{\tau_\level}{h_\level} \| \localU_\level \injectionOp_\level \lambda - \injectionOp_\level \lambda \|_\level^2$) and Lemma \ref{TH:conv_result}:
 \begin{align}
  a_\level (\injectionOp_\level \lambda, \injectionOp_\level \lambda) ~=~ & \| \localQ_\level \injectionOp_\level \lambda \|_0^2 + \frac{\tau_\level}{h_\level} \| \localU_\level \injectionOp_\level \lambda - \injectionOp_\level \lambda \|_\level^2 \lesssim \| \localQ_\level \injectionOp_\level \lambda \|^2_0 \notag\\
  \lesssim & \| \localQ_{\level - 1} \lambda \|_0^2 \le a_{\level - 1}(\lambda, \lambda),
 \end{align}
 while the fourth inequality follows from
 \begin{equation}
  \label{EQ:stability-P}
  \| \projectionOp_{\level-1} \lambda \|^2_{a_{\level-1}} = a_\level(\lambda, \injectionOp_\level \projectionOp_{\level-1} \lambda) \le \| \lambda \|_{a_\level} \| \injectionOp_\level \projectionOp_{\level-1} \lambda \|_{a_\level} \lesssim \| \lambda \|_{a_\level} \| \projectionOp_{\level-1} \lambda \|_{a_{\level-1}},
 \end{equation}
 which itself is a consequence of the Cauchy--Schwarz inequality and the third inequality.
\end{proof}
\begin{lemma}
 Under the assumptions of Lemma \ref{TH:conv_result}, \eqref{EQ:precond2} holds.
\end{lemma}
\begin{proof}
 \begin{align}
  a_\level ( \lambda - & \injectionOp_{\level-1} \projectionOp_\level \lambda, \lambda - \injectionOp_\level \projectionOp_{\level-1} \lambda ) \\
 = & a_\level ( \lambda, \lambda ) - 2 a_\level ( \lambda, \injectionOp_\level \projectionOp_{\level-1} \lambda ) + a_\level ( \injectionOp_\level \projectionOp_{\level-1} \lambda, \injectionOp_\level \projectionOp_{\level-1} \lambda ) \notag\\
 \le & a_\level ( \lambda, \lambda ) \underbrace{ - 2 a_{\level - 1} ( \projectionOp_{\level-1} \lambda,  \projectionOp_{\level-1} \lambda ) }_{ \le 0 } + C \underbrace{ a_{\level-1} ( \projectionOp_{\level-1} \lambda, \projectionOp_{\level-1} \lambda ) }_{ \lesssim \| \lambda \|^2_{a_\level} },\notag
 \end{align}
\end{proof}

\subsection{Proof of \eqref{EQ:precond1}}
%
It remains to show that assumption \eqref{EQ:precond1} holds true, which is the statement of Theorem \ref{TH:A1_proof} below. Its proof follows similar lines as the one for standard multigrid methods, but some adjustments are necessary which make it technically more involved. A typical step in multigrid proofs consists in considering $A_\level \lambda$ as a function in $L^2(\Omega)$ and using it as right hand side in an auxiliary problem.
In the context of HDG methods, $A_\level\lambda$ is only a Borel measure on the skeleton. Since the mapping from $f$ to the ``Neumann'' trace in~\eqref{EQ:hdg_f} cannot be guaranteed to be surjective---simple counting of dimensions in low order cases shows that in some cases it even cannot be surjective---the auxiliary solution $\tilde u$ must be defined in an inconsistent way, introducing an additional error.

Then, $\lambda_\level$ and $\lambda_{\level-1}$ are interpreted as Galerkin approximations to the solution $\tilde u$ of this auxiliary problem, thus obtaining an error estimate through Galerkin orthogonality. Here, we have to deviate in two respects. Here, we only have quasi-orthogonality from Lemma \ref{LEM:quasi_orth}.

We first address the construction of a right hand side in $L^2(\Omega)$ by the lifting $\liftingOp_\level\colon\skeletalSpace_{\level}\to \contThreeElementSpace{\level} \subset L^2(\Omega)$, which was inspired by~\cite{TanPhD} and coincides with~\cite[Lem.A.3]{GiraultR1986} for the two dimensional case and~\cite[Def.5.46]{Monk2003}  for three dimensional. An extremely similar (upto the need of averaging node values) operator has been used to investigate multigrid convergence of EDG methods in \cite{LuRK2021}. For each $\lambda\in\skeletalSpace_\level$ it is defined by the following conditions:
\begin{subequations}
\label{eq:define-s}
\begin{align}
 (\liftingOp_\level \lambda , v)_\elem &= (\localU_\level \lambda, v)_\elem && \forall v \in \polynomials_p(\elem), \, \forall \elem \in \mesh_\level,\\
 \langle \liftingOp_\level \lambda, \eta \rangle_\face &= \langle \lambda, \eta \rangle_\face && \forall \eta \in \polynomials_{p+1}(\face), \, \forall \face \in \faceSet_\level\\
 \liftingOp_\level \lambda(\vec a) &= \avg{ \lambda(\vec a)} && \forall \vec a \text{ is a vertex in } \mesh_\level, \, \vec a \not\in \partial \Omega,\\
 \liftingOp_\level \lambda(\vec a) &= 0 && \forall \vec a \text{ is a vertex in } \mesh_\level, \, \vec a \in \partial \Omega.
\end{align}
\end{subequations}
The degrees of freedom in equations~\eqref{eq:define-s} follow the standard geometrical decomposition of polynomial spaces on simplices, see for instance~\cite[Section 4]{ArnoldFalkWinther06acta}. Thus, it is well-defined and approximates $\localU_\level \lambda$, since it is its $L^2: \contThreeElementSpace{\level} \to \discElementSpace_\level $ projection (with respect to polynomials of degree at most $p$). Moreover, it approximates $\lambda$, since it is its $L^2$ projection, and it approximates the values in the vertices by attending their means. Its construction extends to the three-dimensional case by increasing the polynomial order by one, thus defining $\liftingOp_\level\colon\skeletalSpace_{\level}\to \contFourElementSpace{\level}$ and adding the conditions
\begin{gather}
    \langle \liftingOp_\level \lambda, \eta \rangle_\Gamma =  \langle \avg{\lambda}, \eta \rangle_\Gamma 
    \qquad \forall \eta \in \polynomials_{p+2}(\Gamma),
\end{gather}
for all edges $\Gamma$ of the tetrahedron. Here, $\avg{\lambda}$ is again the average taken over all cells adjacent to $\Gamma$. This construction extends to dimensions higher than three if so desired. Below, we will make use of the range of $\liftingOp_\level$, namely $\liftingOp_\level \skeletalSpace_\level$ and the fact that it is piecewise polynomial.

We summarize the properties of $\liftingOp_\level \lambda$:
\begin{lemma}[Properties of $\liftingOp_\level \lambda$]
 Under assumptions \eqref{EQ:LS1}--\eqref{EQ:LS4}, we have
 \begin{align}
  \| \liftingOp_\level \lambda \|_0 ~ \cong ~ & \| \lambda \|_\level && \forall \lambda \in \skeletalSpace_\level && \text{(norm equiv.)}\label{EQ:lift_equiv}\\
  \liftingOp_\level \gamma_\level w ~=~ & w && \forall w \in \linElementSpace_\level && \text{(lifting identity)}\label{EQ:lift_iden}\\
  |\liftingOp_\level \lambda|_{1} ~\lesssim~ & \| \localQ_\level \lambda \|_0 && \forall \lambda \in \skeletalSpace_\level &&  \text{(lifting bound)}\label{EQ:lift_bound}
 \end{align}
\end{lemma}
\begin{proof}
 By the standard scaling argument and using \eqref{EQ:LS2}, we immediately have norm equivalence.
 
 For all $w \in \linElementSpace_\level$, we have $\localU_\level \gamma_\level w = w$ from \eqref{EQ:LS4}. Thus, by definition of $\liftingOp_\level \gamma_\level w$, we get the lifting identity. Finally, using the lifting identity and inverse inequality yields
 \begin{align}
  | \liftingOp_\level \lambda |_{1} &\le |\liftingOp_\level \lambda - \averagingOp_\level \localU_\level \lambda |_{1} + | \averagingOp_\level \localU_\level \lambda |_{1}\\
  &\lesssim h^{-1}_\level \| \liftingOp_\level \lambda - \liftingOp_\level \gamma_\level \averagingOp_\level \localU_\level \lambda \|_0 + | \averagingOp_\level \localU_\level \lambda |_{1}\\
  &\lesssim h^{-1}_\level \| \lambda - \gamma_\level \averagingOp_\level \localU_\level \lambda \|_\level + | \averagingOp_\level \localU_\level \lambda |_{1},
 \end{align}
 where the last inequality uses norm equivalence~\eqref{EQ:lift_equiv}.
 Thus, we obtain the lifting bound if combined with Lemmas \ref{LEM:avg_bound_by_localQ} \& \ref{LEM:diff_bound_by_q}.
\end{proof}
We now use $\liftingOp_\level$ to construct the right hand side of the auxiliary problem. As $A_\level \lambda$ is in the dual of $\skeletalSpace_\level$, a natural condition is
\begin{equation}\label{EQ:scp_def}
   (f_\lambda, \liftingOp_\level \mu)
   = \langle A_\level \lambda, \mu \rangle
    = a_\level (\lambda, \mu)
   \qquad \forall \mu \in \skeletalSpace_\level.
\end{equation}
This problem is guaranteed to have a unique solution if we search for $f_\lambda$ in the space $\liftingOp_\level \skeletalSpace_\level$.
The construction avoids the question of surjectivity of $\liftingOp_\level$ and will be justified by the error estimates below.
Now, we are ready to define $\tilde u \in H^1_0(\Omega)$ as the unique solution of
\begin{equation}\label{EQ:Pois_app}
   (\nabla\tilde u,\nabla v) = (f_\lambda,v)
 \qquad \forall v\in H^1_0(\Omega).
\end{equation}
Furthermore, for any $\level=0,\dots,L$ we introduce its HDG approximation $\tilde \lambda_\level \in \skeletalSpace_\level$ with
\begin{equation}\label{EQ:tilde_lambda}
 a_\level (\tilde \lambda_\level, \mu) = (f_\lambda, \localU_\level \mu) \qquad \forall \mu \in \skeletalSpace_\level.
\end{equation}
Next, we prove the approximation result for $\tilde \lambda_\level$ which is a generalization of \cite[Lem.~4.6]{LuRK2020}.
\begin{lemma}\label{TH:extract_h_basis}
 Assume that \eqref{EQ:LS1}--\eqref{EQ:LS4} hold. Then, we have for all $\lambda \in \skeletalSpace_\level$
 \begin{equation}
  \label{EQ:bound_flambda}
  \| \lambda - \tilde \lambda_\level \|_{a_\level} \lesssim h_\level \| A_\level \lambda \|_\level, \qquad \| f_\lambda \|_0 \lesssim \| A_\level \lambda \|_\level.
 \end{equation}
\end{lemma}
\begin{proof}
 From \eqref{EQ:scp_def} and \eqref{EQ:lift_equiv}, we have for $\mu = \Phi_\lambda$ that
 \begin{equation}
  \| f_\lambda \|^2_0 = \| \liftingOp_\level \Phi_\lambda \|^2_0 =  \langle A_\level \lambda, \Phi_\lambda \rangle_\level \le \| A_\level \lambda \|_\level \| \Phi_\lambda \|_\level \lesssim \| A_\level \lambda \|_\level \| \liftingOp_\level \Phi_\lambda \|_0,
 \end{equation}
 which implies the second inequality.
 
 Since $a_\level (\lambda, \mu) = (f_\lambda, \liftingOp_\level \mu)$ and $a_\level (\tilde \lambda_\level, \mu) = (f_\lambda, \localU_\level \mu)_0$, we have
 \begin{equation}\label{EQ:lam_diff}
  a_\level (\lambda - \tilde \lambda_\level, \mu) = (f_\lambda, \liftingOp_\level \mu - \localU_\level \mu) \qquad \forall \mu \in \skeletalSpace_\level.
 \end{equation}
 Thus, for all $\mu \in \skeletalSpace_\level$
 \begin{equation}\label{EQ:s_appr}
  \| \liftingOp_\level \mu - \localU_\level \mu \|_0 = \| \liftingOp_\level \mu - \discProj_\level \liftingOp_\level \mu \|_0 \overset{\eqref{EQ:L2H1_approx}}\lesssim h_\level |\liftingOp_\level \mu |_{1} \overset{\eqref{EQ:lift_bound}}\lesssim h_\level \| \localQ_\level \lambda \|_0.
 \end{equation}
 Taking $\mu = \lambda - \tilde \lambda_\level$ in \eqref{EQ:lam_diff} and using \eqref{EQ:s_appr}, we have
 \begin{equation}
  \| \lambda - \tilde \lambda_\level \|^2_{a_\level} \lesssim \| f_\lambda \|_0 h_\level \| \localQ_\level (\lambda - \tilde \lambda_\level) \|_0 \le h_\level \| f_\lambda \|_0 \| \lambda - \tilde \lambda_\level \|_{a_\level},
 \end{equation}
 which yields the result after using the theorem's second inequality.
\end{proof}
\begin{lemma}[Reconstruction approximation]\label{LEM:reconstruction_approx}
 Assume that \eqref{EQ:IA1}, \eqref{EQ:IA2}, as well as \eqref{EQ:LS1}--\eqref{EQ:LS5} hold. If the model problem has elliptic regularity, then for all $\lambda \in \skeletalSpace_\level$, there exists an auxiliary function $\ureconstructed \in \linElementSpace_{\level - 1}$ such that
 \begin{equation}
  \| \localQ_\level \lambda + \nabla \ureconstructed \|_0 + \| \localQ_{\level - 1} \projectionOp_{\level - 1} \lambda + \nabla \ureconstructed \|_0 \lesssim h_\level \| A_\level \lambda \|_\level.
 \end{equation}
\end{lemma}
\begin{proof}
 We only prove the inequality for the second term, since the first one can be treated analogously. We set $\ureconstructed = \contLinProj_{\level - 1} \tilde u$, where $\contLinProj_{\level-1}$ is the $L^2$-projection into $\linElementSpace_{\level-1}$ with the standard stability and approximation estimates
 \begin{xalignat}2
      | \contLinProj_\level u |_{1} &\lesssim  |u|_{1},  && \forall u \in H^1(\Omega),\label{EQ:L2H1_stability}\\
 \| u - \contLinProj_\level u \|_0 &\lesssim h_\level^2 |u|_2,  && \forall u \in H^2(\Omega).\label{EQ:L2H2_approx}
 \end{xalignat}
 
 and split up the term, we want to bound via
 \begin{multline}
  \| \localQ_{\level - 1} \projectionOp_{\level - 1} \lambda + \nabla \ureconstructed \|_0
  \le \underbrace{ \| \localQ_{\level - 1} \projectionOp_{\level - 1} \lambda - \localQ_{\level - 1} \projectionOp_{\level - 1} \tilde \lambda_\level \|_0 }_{ =: \Xi_1 }\\
   + \underbrace{ \| \localQ_{\level - 1} \projectionOp_{\level - 1} \tilde \lambda_\level - \localQ_{\level - 1} \tilde \lambda_{\level-1} \|_0 }_{ =: \Xi_2 }
  + \underbrace{ \| \localQ_{\level - 1} \tilde \lambda_{\level-1} + \nabla \contLinProj_{\level - 1} \tilde u \|_0 }_{ =: \Xi_3 }.
 \end{multline}
 The estimate of $\Xi_1$ follows immediately from \eqref{EQ:bound_flambda}, \eqref{EQ:stability-P} and \eqref{EQ:bilinear_condensed}. Using \eqref{EQ:projection_definition}, \eqref{EQ:tilde_lambda}, \eqref{EQ:bound_flambda}, Lemma \ref{TH:conv_result}, and setting $e_{\level-1} = \projectionOp_{\level-1} \tilde \lambda_\level - \tilde \lambda_{\level-1}$  we obtain 
 \begin{equation}
  \Xi_2^2 \le a_{\level-1}( e_{\level-1}, e_{\level-1} ) \le \| f_\lambda \|_0 \| \localU_{\level-1} e_{\level-1} - \localU_\level \injectionOp_\level e_{\level -1} \|_0 \lesssim h_\level \Xi_2\| A_\level \lambda \|_\level.
 \end{equation}
 
 Finally, we rewrite $\Xi_3$ by \eqref{EQ:LS4}, estimate it using \eqref{EQ:LS2}
 \begin{align}
  \Xi_3 = & \| \localQ_{\level - 1} \tilde \lambda_{\level-1} - \localQ_{\level - 1} \gamma_{\level - 1} \contLinProj_{\level - 1} \tilde u \|_0 \\
  \lesssim & h_{\level - 1}^{-1} \| \tilde \lambda_{\level-1} - \gamma_{\level - 1} \contLinProj_{\level - 1} \tilde u \|_{\level - 1}\notag \\
  \le~ & h_{\level - 1}^{-1} \left(\| \tilde \lambda_{\level-1} - \skeletalProj_{\level - 1} \tilde u \|_{\level - 1} + \| \skeletalProj_{\level - 1} \tilde u - \tilde u \|_{\level-1} + \| \tilde u - \contLinProj_{\level - 1} \tilde u \|_0\right), \notag
 \end{align}
 and use triangle's inequality. Next, we use \eqref{EQ:LS5} on the first summand, \eqref{EQ:bdrH2_approx} on the second summand, and \eqref{EQ:L2H2_approx} on the third summand to obtain
 \begin{align}
  \Xi_3 \lesssim & h_{\level - 1}^{-1} \left( h_{\level - 1}^2 |\tilde u|_2 \right) \lesssim h_{\level - 1} \| A_\level \lambda \|_\level \lesssim h_\level \| A_\level \lambda \|_\level \notag
 \end{align}
\end{proof}
Using the aforementioned results, we can formulate an improved version of \cite[Theo.~4.1]{LuRK2020}:
\begin{theorem}\label{TH:A1_proof}
 If \eqref{EQ:diffusion_mixed} has elliptic regularity, $\tau_\level h_\level \lesssim 1$, and $\eqref{EQ:IA1}$, $\eqref{EQ:IA2}$, and \eqref{EQ:LS1}--\eqref{EQ:LS6} hold, then \eqref{EQ:precond1} is satisfied.
\end{theorem}
\begin{proof}
 The proof uses similar arguments as \cite[Theo.~4.1]{LuRK2020}. Nevertheless, since there are a few modifications and this article is more general, we present it here as well.
 First, we note that by standard arguments, for instance~\cite[Theorem 3.6]{DuanGTZ2007}, it is sufficient to prove
 \begin{gather}
     \left| a_\level(\lambda - \injectionOp_\level \projectionOp_{\level-1} \lambda, \lambda) \right| \lesssim h^2_\level \| A_\level \lambda \|^2_\level
 \end{gather}
 By the bilinearity of $a_\level(.,.)$ and its definition in~\eqref{EQ:bilinear_condensed}, there holds
  \begin{align}
  a_\level(\lambda - \injectionOp_\level& \projectionOp_{\level-1} \lambda, \lambda) = a_\level(\lambda, \lambda) - a_{\level-1}(\projectionOp_{\level-1} \lambda, \projectionOp_{\level-1} \lambda) \notag\\
  \cong & (\localQ_\level \lambda, \localQ_\level \lambda)_0 - (\localQ_{\level-1} \projectionOp_{\level-1} \lambda, \localQ_{\level-1} \projectionOp_{\level-1} \lambda)_0 \tag{T1} \label{EQ:main_proof_T1}\\
  & + \frac{\tau_\level}{h_\level} \| \localU_\level \lambda - \lambda \|^2_\level - \frac{\tau_{\level-1}}{h_{\level-1}} \| \localU_{\level-1} \projectionOp_{\level-1} \lambda - \projectionOp_{\level-1} \lambda \|^2_{\level-1}. \tag{T2} \label{EQ:main_proof_T2}
 \end{align}
By binomial factorization, we obtain
 \begin{gather}
   \eqref{EQ:main_proof_T1}
  =  (\localQ_\level \lambda + \localQ_{\level-1} \projectionOp_{\level-1} \lambda, \localQ_\level \lambda - \localQ_{\level-1} \projectionOp_{\level-1} \lambda)_0.
 \end{gather}
Now we use quasi-orthogonality on the left to insert $\nabla w$ with $w\in \linElementSpace_{\level-1}$ and add zero on the right to obtain
\begin{gather}
   \eqref{EQ:main_proof_T1}
   =  (\localQ_\level \lambda + 2 \nabla w + \localQ_{\level-1} \projectionOp_{\level-1} \lambda, \localQ_\level \lambda \pm \nabla w - \localQ_{\level-1} \projectionOp_{\level-1} \lambda)_0.
\end{gather}
Thus, applying Bunyakovsky-Cauchy-Schwarz inequality and the reconstruction approximation lemma~\ref{LEM:reconstruction_approx} yields
\begin{gather}
   \eqref{EQ:main_proof_T1}
   \lesssim h_\level \| A_\level \lambda \|_\level
   \; h_\level \| A_\level \lambda \|_\level
    = h_\level^2 \| A_\level \lambda \|_\level^2
   .
\end{gather}

It remains to show an equivalent estimate for~\eqref{EQ:main_proof_T2}. We estimate the first term using \eqref{EQ:LS4} for the identities and \eqref{EQ:LS1} for the estimate:
\begin{align}
  \frac{\tau_\level}{h_\level} \| \localU_\level \lambda - \lambda \|^2_\level = & \frac{\tau_\level}{h_\level} \| \localU_\level (\lambda - \gamma_\level w) + (\lambda - \gamma_\level w) \|^2_\level\\
  ~\lesssim~ & \tau_\level h_\level \| \localQ_\level (\lambda - \gamma_\level w) \|^2_0 = \tau_\level h_\level \| \localQ_\level \lambda + \nabla w \|^2_0,
\end{align}
for any $w\in \linElementSpace_{\level}$. Again, we use lemma~\ref{LEM:reconstruction_approx} to estimate
\begin{gather}
    \frac{\tau_\level}{h_\level} \| \localU_\level \lambda - \lambda \|^2_\level
    \lesssim h_\level^2 \| A_\level \lambda \|_\level^2.
\end{gather}
The second term has the same structure on level $\level-1$. Exploiting the fact that $h_\level$ and $h_{\level-1}$ are similar due to~\eqref{EQ:cref}, we can use the same argument together with boundedness of $\projectionOp_{\level-1}$.
\end{proof}
%
\section{HDG methods covered by our analysis}\label{SEC:possible_methods}
%
In this section, we give three HDG methods sufficing the general assumptions on local problems:
\subsection{LDG-H methods} 
Here, the spaces are
\begin{gather}
    \discElementSpace_\elem = \polynomials_p(\elem)
    \qquad\vec W_\elem = [\polynomials_p(\elem)]^d
\end{gather}
Under the assumption $\tau_\level h_\level \lesssim 1$, the required results have been proven in the following sources: 
\begin{itemize}
 \item \eqref{EQ:LS1} is a combination of \cite[(3.10) \& (3.14)]{CockburnDGT2013}.
 \item \eqref{EQ:LS2} can be found in \cite[Theo.~3.1]{CockburnDGT2013}.
 \item \eqref{EQ:LS3} is \cite[Lem.~3.3]{ChenLX2014}.
 \item \eqref{EQ:LS4} is \cite[Lem.~3.5.iii]{CockburnDGT2013}.
 \item \eqref{EQ:LS5} is \cite[Thm.~4.1]{CockburnGS2010}.
 \item \eqref{EQ:LS6} is \cite[Thm.~4.8]{TanPhD}.
\end{itemize}
\subsection{RT-H methods}
Here, $\tau_\level \equiv 0$ and we use the standard Raviart-Thomas spaces
\begin{gather}
    \discElementSpace_\elem = \polynomials_p(\elem)
    \qquad
    \vec W_\elem = [\polynomials_p(\elem)]^d + \vec x \polynomials_p
\end{gather}
Proof of the assumptions can be found at:
\begin{itemize}
 \item \eqref{EQ:LS1} is \cite[Lem.~4.2]{TanPhD}.
 \item \eqref{EQ:LS2} is \cite[Lem.~3.3]{CockburnG2005}.
 \item \eqref{EQ:LS3} is similar to \cite[Lem.~3.3]{ChenLX2014}.
 \item \eqref{EQ:LS4} is easy to verify.
 \item \eqref{EQ:LS5} is \cite[Cor.~3.9]{CockburnG2005}.
 \item \eqref{EQ:LS6} is \cite[Thm.~2.3]{Gopa2003}.
\end{itemize}
\subsection{BDM-H methods} $\discElementSpace_\elem = \polynomials_{p-1}(\elem)$, $\vec W_\elem = [\polynomials_p(\elem)]^d$, $\tau_\level \equiv 0$\\
Here, we restrict ourselves to the cases, where $p \ge 2$:
\begin{itemize}
 \item \eqref{EQ:LS1} is similar to the RT-H case. Thus, integrating \eqref{EQ:hdg_primary} by parts on $\elem$ gives
 \begin{equation}\label{EQ:BDM_proof}
  (\localQ^\textup{BDM}_\level \lambda + \nabla \localU^\textup{BDM}_\level \lambda, \vec p_\elem)_\elem \cong h^{-1}_\level \langle \localU^\textup{BDM}_\level \lambda - \lambda, \vec p_\elem \cdot \Nu_\elem \rangle_{\partial \elem}
 \end{equation}
 According to \cite[(3.41)]{BrezziF2012}, there is a $\vec p_\elem$ such that
 \begin{gather*}
  \vec p_\elem \cdot \Nu_\elem = \localU^\textup{BDM}_\level \lambda - \lambda, \qquad (\vec p_\elem, \nabla w)_{0,\elem} = 0 \quad \forall w \in \polynomials_{p-1}(\elem),\\
  \text{ and } \quad \| \vec p_\elem \|_{0,\elem} \lesssim \| \vec p_\elem \|_{\level, \partial \elem},
 \end{gather*}
 which is obvious from the degrees of freedom of BDM and a scaling argument. With this $\vec p_\elem$ in \eqref{EQ:BDM_proof}, we obtain
 \begin{gather*}
  h^{-1}_\level \| \localU^\textup{BDM}_\level \lambda - \lambda \|^2_{\level, \partial \elem} \cong (\localQ^\textup{BDM}_\level \lambda , \vec p_\elem)_\elem \le \| \localQ^\textup{BDM}_\level \lambda \|_{0,\elem} \| \vec p_\elem \|_{0,\elem}\\
  \lesssim \| \localQ^\textup{BDM}_\level \lambda \|_{0,\elem} \| \localU^\textup{BDM}_\level \lambda - \lambda \|_{\level, \partial\elem}
 \end{gather*}
 \item \eqref{EQ:LS3} is similar to \cite[Lem.~3.3]{ChenLX2014}.
 \item \eqref{EQ:LS4} is easy to verify.
 \item \eqref{EQ:LS5} is \cite[Cor.~7.1]{CockburnG2005}.
 \item \eqref{EQ:LS2} \& \eqref{EQ:LS6} can be deduced by the comparison of RT-H and BDM-H. By \cite[Lem.~4.4]{CockburnG2004}, we have on each $\elem \in \mesh_\level$ that
 \begin{equation*}
  \localQ_\level^\textup{RT} \lambda = \localQ_\level^\textup{BDM} \lambda, \qquad \localU^\textup{BDM}_\level \lambda = \discProj_{\level,p-1} \localU^\textup{RT}_\level \lambda,
 \end{equation*}
 where $\discProj_{\level,p-1} \localU^\textup{RT}_\level \lambda \in \polynomials_{p-1}(\elem)$ satisfying
 \begin{equation*}
  (\discProj_{\level,p-1} \localU^\textup{RT}_\level \lambda, v)_0 = (\localU^\textup{RT}_\level \lambda, v) \qquad \forall v \in \polynomials_{p-1}(\elem).
 \end{equation*}
 Then, \eqref{EQ:LS2} for RT-H implies \eqref{EQ:LS2} for BDM-H and by \cite[Thm.~4.1]{CockburnG2004} the same holds for \eqref{EQ:LS6}.
\end{itemize}
%
\section{Numerical experiments}\label{SEC:numerics}
%
\begin{figure}[tp]
 \begin{tikzpicture}[scale = 2.]
  \draw (0,2) -- (0,0) -- (1,0) -- (1,1) -- (0,1) -- (1,0) -- (2,0) -- (2,2) -- (1,2) -- (1,1) -- (2,1) -- (1,2) -- (0,2) -- (2,0);
 \end{tikzpicture}
 \caption{Initial mesh for numerical experiments.}\label{FIG:mesh}
\end{figure}
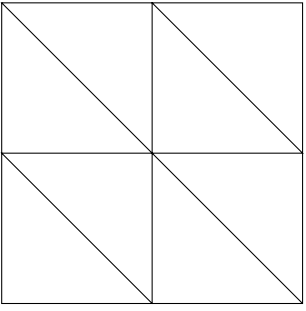
For the numerical evaluation of our multigrid method for HDG, we consider the following Poisson problem on the unit square $\Omega = [0,1]^2$:
\begin{subequations}\label{EQ:num_testcase}
\begin{align}
 -\Delta u & = 1 && \text{ in } \Omega,\\
 u & = 0 && \text{ on } \partial \Omega.
\end{align}
\end{subequations}
The first mesh is shown in Figure \ref{FIG:mesh} and it is successively refined in our experiments. The implementation is based on the FFW toolbox from \cite{BGGRW07}. It uses a Lagrange basis with equidistant support points and the Euclidean inner product in the coefficient space instead of the inner product $\langle .,. \rangle_\level$. These two inner products are equivalent up to a factor of $h^2_\level$.
Supposing that the matrix form of \eqref{EQ:num_testcase} is $A \vec x = \vec b$, we stop the iteration for solving the linear system of equations if
\begin{equation}
 \frac{ \| \vec b - A \vec x_\text{iter} \|_2 }{ \| \vec b \|_2 } < 10^{-6}.
\end{equation}
The results of this procedure are illustrated in Tables \ref{TAB:multigrid_steps_i0} -- \ref{TAB:multigrid_steps_i3}.
\begin{table}
 \begin{tabular}{cc|@{\,}lcc@{\,}lcc@{\,}lcc@{\,}lcc@{\,}lcc@{\,}lcc}
  \toprule
  \multicolumn{2}{c|@{\,}}{mesh level}  && \multicolumn{2}{c}{2}  && \multicolumn{2}{c}{3}   && \multicolumn{2}{c}{4}   && \multicolumn{2}{c}{5}    && \multicolumn{2}{c}{6}     && \multicolumn{2}{c}{7}     \\
  \cmidrule{4-5} \cmidrule{7-8} \cmidrule{10-11} \cmidrule{13-14} \cmidrule{16-17} \cmidrule{19-20}
  \multicolumn{2}{c|@{\,}}{smoother}    && 1  & 2  && 1  & 2  && 1  & 2  && 1  & 2  && 1  & 2  && 1  & 2  \\
  \midrule
  \multirow{3.5}{*}{\rotatebox[origin=c]{90}{$p = 1$}}
  & \# DoFs               && \multicolumn{2}{c}{80} && \multicolumn{2}{c}{352} && \multicolumn{2}{c}{1472} && \multicolumn{2}{c}{6016} && \multicolumn{2}{c}{24320} && \multicolumn{2}{c}{97792} \\
  \cmidrule{4-5} \cmidrule{7-8} \cmidrule{10-11} \cmidrule{13-14} \cmidrule{16-17} \cmidrule{19-20}
  & $\tau = \tfrac{1}{h}$ && 33 & 17 && 39 & 20 && 38 & 19 && 36 & 19 && 35 & 18 && 35 & 18 \\
  & $\tau = 1$            && 33 & 17 && 39 & 19 && 36 & 18 && 35 & 18 && 34 & 17 && 33 & 17 \\
  \midrule
  \multirow{3.5}{*}{\rotatebox[origin=c]{90}{$p = 2$}}
  & \# DoFs               && \multicolumn{2}{c}{120} && \multicolumn{2}{c}{528} && \multicolumn{2}{c}{2208} && \multicolumn{2}{c}{9024} && \multicolumn{2}{c}{36480} && \multicolumn{2}{c}{146688} \\
  \cmidrule{4-5} \cmidrule{7-8} \cmidrule{10-11} \cmidrule{13-14} \cmidrule{16-17} \cmidrule{19-20}
  & $\tau = \tfrac{1}{h}$ && 13 & 08 && 12 & 07 && 11 & 07 && 10 & 06 && 10 & 06 && 09 & 05 \\
  & $\tau = 1$            && 13 & 08 && 12 & 07 && 11 & 07 && 10 & 06 && 10 & 06 && 09 & 05 \\
  \midrule
  \multirow{3.5}{*}{\rotatebox[origin=c]{90}{$p = 3$}}
  & \# DoFs               && \multicolumn{2}{c}{160} && \multicolumn{2}{c}{704} && \multicolumn{2}{c}{2944} && \multicolumn{2}{c}{12032} && \multicolumn{2}{c}{48640} && \multicolumn{2}{c}{195584} \\
  \cmidrule{4-5} \cmidrule{7-8} \cmidrule{10-11} \cmidrule{13-14} \cmidrule{16-17} \cmidrule{19-20}
  & $\tau = \tfrac{1}{h}$ && 24 & 15 && 25 & 15 && 25 & 15 && 25 & 15 && 25 & 15 && 25 & 15 \\
  & $\tau = 1$            && 24 & 15 && 25 & 15 && 25 & 15 && 25 & 15 && 25 & 15 && 25 & 15 \\
  \bottomrule
 \end{tabular}\vspace{1ex}
 \caption{Numbers of iterations with one and two smoothing steps. The polynomial degree of the HDG method is $p$ and the injection operator is the one from \cite{LuRK2020}.}\label{TAB:multigrid_steps_i0}
\end{table}
\begin{table}
 \begin{tabular}{cc|@{\,}lcc@{\,}lcc@{\,}lcc@{\,}lcc@{\,}lcc@{\,}lcc}
  \toprule
  \multicolumn{2}{c|@{\,}}{mesh level}  && \multicolumn{2}{c}{2}  && \multicolumn{2}{c}{3}   && \multicolumn{2}{c}{4}   && \multicolumn{2}{c}{5}    && \multicolumn{2}{c}{6}     && \multicolumn{2}{c}{7}     \\
  \cmidrule{4-5} \cmidrule{7-8} \cmidrule{10-11} \cmidrule{13-14} \cmidrule{16-17} \cmidrule{19-20}
  \multicolumn{2}{c|@{\,}}{smoother}    && 1  & 2  && 1  & 2  && 1  & 2  && 1  & 2  && 1  & 2  && 1  & 2  \\
  \midrule
  \multirow{3.5}{*}{\rotatebox[origin=c]{90}{$p = 1$}}
  & \# DoFs               && \multicolumn{2}{c}{80} && \multicolumn{2}{c}{352} && \multicolumn{2}{c}{1472} && \multicolumn{2}{c}{6016} && \multicolumn{2}{c}{24320} && \multicolumn{2}{c}{97792} \\
  \cmidrule{4-5} \cmidrule{7-8} \cmidrule{10-11} \cmidrule{13-14} \cmidrule{16-17} \cmidrule{19-20}
  & $\tau = \tfrac{1}{h}$ && 18 & 10 && 22 & 12 && 22 & 12 && 23 & 12 && 23 & 12 && 23 & 12 \\
  & $\tau = 1$            && 18 & 10 && 21 & 12 && 22 & 12 && 22 & 12 && 22 & 12 && 23 & 12 \\
  \midrule
  \multirow{3.5}{*}{\rotatebox[origin=c]{90}{$p = 2$}}
  & \# DoFs               && \multicolumn{2}{c}{120} && \multicolumn{2}{c}{528} && \multicolumn{2}{c}{2208} && \multicolumn{2}{c}{9024} && \multicolumn{2}{c}{36480} && \multicolumn{2}{c}{146688} \\
  \cmidrule{4-5} \cmidrule{7-8} \cmidrule{10-11} \cmidrule{13-14} \cmidrule{16-17} \cmidrule{19-20}
  & $\tau = \tfrac{1}{h}$ && 13 & 08 && 13 & 07 && 12 & 07 && 12 & 07 && 12 & 07 && 12 & 07 \\
  & $\tau = 1$            && 13 & 08 && 13 & 07 && 12 & 07 && 12 & 07 && 12 & 07 && 12 & 07 \\
  \midrule
  \multirow{3.5}{*}{\rotatebox[origin=c]{90}{$p = 3$}}
  & \# DoFs               && \multicolumn{2}{c}{160} && \multicolumn{2}{c}{704} && \multicolumn{2}{c}{2944} && \multicolumn{2}{c}{12032} && \multicolumn{2}{c}{48640} && \multicolumn{2}{c}{195584} \\
  \cmidrule{4-5} \cmidrule{7-8} \cmidrule{10-11} \cmidrule{13-14} \cmidrule{16-17} \cmidrule{19-20}
  & $\tau = \tfrac{1}{h}$ && 17 & 11 && 17 & 10 && 17 & 10 && 17 & 10 && 17 & 10 && 17 & 10 \\
  & $\tau = 1$            && 17 & 11 && 17 & 10 && 17 & 10 && 17 & 10 && 17 & 10 && 17 & 10 \\
  \bottomrule
 \end{tabular}\vspace{1ex}
 \caption{Numbers of iterations with one and two smoothing steps. The polynomial degree of the HDG method is $p$ and the injection operator is $\injectionOp_\level^1$.}\label{TAB:multigrid_steps_i1}
\end{table}
\begin{table}
 \begin{tabular}{cc|@{\,}lcc@{\,}lcc@{\,}lcc@{\,}lcc@{\,}lcc@{\,}lcc}
  \toprule
  \multicolumn{2}{c|@{\,}}{mesh level}  && \multicolumn{2}{c}{2}  && \multicolumn{2}{c}{3}   && \multicolumn{2}{c}{4}   && \multicolumn{2}{c}{5}    && \multicolumn{2}{c}{6}     && \multicolumn{2}{c}{7}     \\
  \cmidrule{4-5} \cmidrule{7-8} \cmidrule{10-11} \cmidrule{13-14} \cmidrule{16-17} \cmidrule{19-20}
  \multicolumn{2}{c|@{\,}}{smoother}    && 1  & 2  && 1  & 2  && 1  & 2  && 1  & 2  && 1  & 2  && 1  & 2  \\
  \midrule
  \multirow{3.5}{*}{\rotatebox[origin=c]{90}{$p = 1$}}
  & \# DoFs               && \multicolumn{2}{c}{80} && \multicolumn{2}{c}{352} && \multicolumn{2}{c}{1472} && \multicolumn{2}{c}{6016} && \multicolumn{2}{c}{24320} && \multicolumn{2}{c}{97792} \\
  \cmidrule{4-5} \cmidrule{7-8} \cmidrule{10-11} \cmidrule{13-14} \cmidrule{16-17} \cmidrule{19-20}
  & $\tau = \tfrac{1}{h}$ && 18 & 10 && 22 & 12 && 22 & 12 && 23 & 12 && 23 & 12 && 23 & 12 \\
  & $\tau = 1$            && 18 & 10 && 21 & 12 && 22 & 12 && 22 & 12 && 22 & 12 && 23 & 12 \\
  \midrule
  \multirow{3.5}{*}{\rotatebox[origin=c]{90}{$p = 2$}}
  & \# DoFs               && \multicolumn{2}{c}{120} && \multicolumn{2}{c}{528} && \multicolumn{2}{c}{2208} && \multicolumn{2}{c}{9024} && \multicolumn{2}{c}{36480} && \multicolumn{2}{c}{146688} \\
  \cmidrule{4-5} \cmidrule{7-8} \cmidrule{10-11} \cmidrule{13-14} \cmidrule{16-17} \cmidrule{19-20}
  & $\tau = \tfrac{1}{h}$ && 11 & 08 && 11 & 07 && 11 & 07 && 11 & 07 && 11 & 07 && 11 & 07 \\
  & $\tau = 1$            && 11 & 08 && 11 & 07 && 11 & 07 && 11 & 07 && 11 & 07 && 11 & 07 \\
  \midrule
  \multirow{3.5}{*}{\rotatebox[origin=c]{90}{$p = 3$}}
  & \# DoFs               && \multicolumn{2}{c}{160} && \multicolumn{2}{c}{704} && \multicolumn{2}{c}{2944} && \multicolumn{2}{c}{12032} && \multicolumn{2}{c}{48640} && \multicolumn{2}{c}{195584} \\
  \cmidrule{4-5} \cmidrule{7-8} \cmidrule{10-11} \cmidrule{13-14} \cmidrule{16-17} \cmidrule{19-20}
  & $\tau = \tfrac{1}{h}$ && 17 & 11 && 17 & 10 && 17 & 10 && 17 & 10 && 17 & 10 && 17 & 10 \\
  & $\tau = 1$            && 17 & 11 && 17 & 10 && 17 & 10 && 17 & 10 && 17 & 10 && 17 & 10 \\
  \bottomrule
 \end{tabular}\vspace{1ex}
 \caption{Numbers of iterations with one and two smoothing steps. The polynomial degree of the HDG method is $p$ and the injection operator is $\injectionOp_\level^2$.}\label{TAB:multigrid_steps_i2}
\end{table}
\begin{table}
 \begin{tabular}{cc|@{\,}lcc@{\,}lcc@{\,}lcc@{\,}lcc@{\,}lcc@{\,}lcc}
  \toprule
  \multicolumn{2}{c|@{\,}}{mesh level}  && \multicolumn{2}{c}{2}  && \multicolumn{2}{c}{3}   && \multicolumn{2}{c}{4}   && \multicolumn{2}{c}{5}    && \multicolumn{2}{c}{6}     && \multicolumn{2}{c}{7}     \\
  \cmidrule{4-5} \cmidrule{7-8} \cmidrule{10-11} \cmidrule{13-14} \cmidrule{16-17} \cmidrule{19-20}
  \multicolumn{2}{c|@{\,}}{smoother}    && 1  & 2  && 1  & 2  && 1  & 2  && 1  & 2  && 1  & 2  && 1  & 2  \\
  \midrule
  \multirow{3.5}{*}{\rotatebox[origin=c]{90}{$p = 1$}}
  & \# DoFs               && \multicolumn{2}{c}{80} && \multicolumn{2}{c}{352} && \multicolumn{2}{c}{1472} && \multicolumn{2}{c}{6016} && \multicolumn{2}{c}{24320} && \multicolumn{2}{c}{97792} \\
  \cmidrule{4-5} \cmidrule{7-8} \cmidrule{10-11} \cmidrule{13-14} \cmidrule{16-17} \cmidrule{19-20}
  & $\tau = \tfrac{1}{h}$ && 18 & 10 && 22 & 12 && 22 & 12 && 23 & 12 && 23 & 12 && 23 & 12 \\
  & $\tau = 1$            && 18 & 10 && 21 & 12 && 22 & 12 && 22 & 12 && 22 & 12 && 23 & 12 \\
  \midrule
  \multirow{3.5}{*}{\rotatebox[origin=c]{90}{$p = 2$}}
  & \# DoFs               && \multicolumn{2}{c}{120} && \multicolumn{2}{c}{528} && \multicolumn{2}{c}{2208} && \multicolumn{2}{c}{9024} && \multicolumn{2}{c}{36480} && \multicolumn{2}{c}{146688} \\
  \cmidrule{4-5} \cmidrule{7-8} \cmidrule{10-11} \cmidrule{13-14} \cmidrule{16-17} \cmidrule{19-20}
  & $\tau = \tfrac{1}{h}$ && 13 & 08 && 13 & 07 && 12 & 07 && 12 & 07 && 12 & 07 && 12 & 07 \\
  & $\tau = 1$            && 13 & 08 && 13 & 07 && 12 & 07 && 12 & 07 && 12 & 07 && 12 & 07 \\
  \midrule
  \multirow{3.5}{*}{\rotatebox[origin=c]{90}{$p = 3$}}
  & \# DoFs               && \multicolumn{2}{c}{160} && \multicolumn{2}{c}{704} && \multicolumn{2}{c}{2944} && \multicolumn{2}{c}{12032} && \multicolumn{2}{c}{48640} && \multicolumn{2}{c}{195584} \\
  \cmidrule{4-5} \cmidrule{7-8} \cmidrule{10-11} \cmidrule{13-14} \cmidrule{16-17} \cmidrule{19-20}
  & $\tau = \tfrac{1}{h}$ && 17 & 11 && 17 & 10 && 17 & 10 && 17 & 10 && 17 & 10 && 17 & 10 \\
  & $\tau = 1$            && 17 & 11 && 17 & 10 && 17 & 10 && 17 & 10 && 17 & 10 && 17 & 10 \\
  \bottomrule
 \end{tabular}\vspace{1ex}
 \caption{Numbers of iterations with one and two smoothing steps for $f \equiv 1$. The polynomial degree of the HDG method is $p$ and the injection operator is $\injectionOp_\level^3$.}\label{TAB:multigrid_steps_i3}
\end{table}
%
\section{Conclusions}
%
In this paper, we have extended the convergence analysis of the homogeneous multigrid method for HDG \cite{LuRK2020} to more general cases. The stabilization parameter $\tau_\level$ which need to be $\tau_\level = \tfrac{c}{h_\level}$ in \cite{LuRK2020} has been generalized to $\tau_\level h_\level \lesssim 1$ without influencing the convergence of the multigrid method. Furthermore, the injection operator used in \cite{LuRK2020} can be replaced by any injection operator satisfying \eqref{EQ:IA1} and \eqref{EQ:IA2}. Moreover, the theoretical analysis also covers RT-H and BDM-H.
\bibliographystyle{siamplain}
\bibliography{MultigridInjection}

\begin{thebibliography}{10}

\bibitem{ArnoldFalkWinther06acta}
{\sc D.~N. Arnold, R.~S. Falk, and R.~Winther}, {\em Finite element exterior
  calculus, homological techniques, and applications}, Acta Numerica, 15
  (2006), pp.~1--155, \url{https://doi.org/10.1017/S0962492906210018}.

\bibitem{BrambleP1992}
{\sc J.~Bramble and J.~Pasciak}, {\em The analysis of smoothers for multigrid
  algorithms}, Mathematics of Computation, 58 (1992), pp.~467--488,
  \url{https://doi.org/10.1090/S0025-5718-1992-1122058-0}.

\bibitem{BramblePX1991}
{\sc J.~Bramble, J.~Pasciak, and J.~Xu}, {\em The analysis of multigrid
  algorithms with nonnested spaces or noninherited quadratic forms},
  Mathematics of Computation, 56 (1991), pp.~1--34,
  \url{http://www.jstor.org/stable/2008527}.

\bibitem{BrezziF2012}
{\sc F.~Brezzi and M.~Fortin}, {\em Mixed and hybrid finite element methods},
  vol.~15, Springer Science \& Business Media, 2012.

\bibitem{BGGRW07}
{\sc A.~Byfut, J.~Gedicke, D.~Günther, J.~Reininghaus, and S.~Wiedemann}, {\em
  {FFW} documentation}.
\newblock https://github.com/project-openffw/openffw.

\bibitem{ChenLX2014}
{\sc H.~Chen, P.~Lu, and X.~Xu}, {\em A robust multilevel method for
  hybridizable discontinuous {G}alerkin method for the {H}elmholtz equation},
  Journal of Computational Physics, 264 (2014), pp.~133--151,
  \url{https://doi.org/10.1016/j.jcp.2014.01.042},
  \url{http://www.sciencedirect.com/science/article/pii/S0021999114000801}.

\bibitem{CockburnDGT2013}
{\sc B.~Cockburn, O.~Dubois, J.~Gopalakrishnan, and S.~Tan}, {\em Multigrid for
  an {HDG} method}, IMA Journal of Numerical Analysis, 34 (2013),
  pp.~1386--1425, \url{https://doi.org/10.1093/imanum/drt024}.

\bibitem{CockburnG2004}
{\sc B.~Cockburn and J.~Gopalakrishnan}, {\em A characterization of hybridized
  mixed methods for second order elliptic problems}, SIAM Journal on Numerical
  Analysis, 42 (2004), pp.~283--301.

\bibitem{CockburnG2005}
{\sc B.~Cockburn and J.~Gopalakrishnan}, {\em Error analysis of variable degree
  mixed methods for elliptic problems via hybridization}, Mathematics of
  computation, 74 (2005), pp.~1653--1677.

\bibitem{CockburnGL2009}
{\sc B.~Cockburn, J.~Gopalakrishnan, and R.~Lazarov}, {\em Unified
  hybridization of discontinuous {G}alerkin, mixed, and continuous {G}alerkin
  methods for second order elliptic problems}, SIAM Journal on Numerical
  Analysis, 47 (2009), pp.~1319--1365, \url{https://doi.org/10.1137/070706616}.

\bibitem{CockburnGS2010}
{\sc B.~Cockburn, J.~Gopalakrishnan, and F.-J. Sayas}, {\em A projection-based
  error analysis of {HDG} methods}, Mathematics of Computation, 79 (2010),
  pp.~1351--1367.

\bibitem{DuanGTZ2007}
{\sc H.~Duan, S.~Gao, R.~Tan, and S.~Zhang}, {\em A generalized {BPX} multigrid
  framework covering nonnested {V}-cycle methods}, Mathematics of Computation,
  76 (2007), pp.~137--152, \url{http://www.jstor.org/stable/40234371}.

\bibitem{FabienKMR19}
{\sc M.~S. Fabien, M.~G. Knepley, R.~T. Mills, and B.~M. Rivière}, {\em
  Manycore parallel computing for a hybridizable discontinuous {G}alerkin
  nested multigrid method}, SIAM Journal on Scientific Computing, 41 (2019),
  pp.~C73--C96, \url{https://doi.org/10.1137/17M1128903},
  \url{https://doi.org/10.1137/17M1128903}.

\bibitem{GiraultR1986}
{\sc V.~Girault and P.~Raviart}, {\em Finite Element Methods for Navier-Stokes
  Equations}, Springer-Verlag, Berlin Heidelberg, 1986,
  \url{https://doi.org/10.1007/978-3-642-61623-5}.

\bibitem{Gopa2003}
{\sc J.~Gopalakrishnan}, {\em A {S}chwarz preconditioner for a hybridized mixed
  method}, Computational Methods in Applied Mathematics, 3 (2003),
  pp.~116--134, \url{https://doi.org/10.2478/cmam-2003-0009}.

\bibitem{LuRK2020}
{\sc P.~Lu, A.~Rupp, and G.~Kanschat}, {\em {HMG} --- {H}omogeneous multigrid
  for {HDG}}, 2020, \url{https://arxiv.org/abs/2011.14018}.
\newblock arXiv preprint arXiv:2011.14018.

\bibitem{LuRK2021}
{\sc P.~Lu, A.~Rupp, and G.~Kanschat}, {\em Homogeneous multigrid for embedded
  discontinuous {G}alerkin methods}, 2021,
  \url{https://arxiv.org/abs/2101.12645}.
\newblock arXiv preprint arXiv:2101.12645.

\bibitem{Monk2003}
{\sc P.~Monk}, {\em Finite Element Methods for Maxwell’s Equations}, Oxford
  University Press, New York, 2003.

\bibitem{TanPhD}
{\sc S.~Tan}, {\em Iterative solvers for hybridized finite element methods},
  PhD thesis, University of Florida, 2009,
  \url{http://etd.fcla.edu/UF/UFE0024820/tan_s.pdf}.

\end{thebibliography}
\end{document}